\documentclass[11pt]{elsarticle}
\usepackage{amsmath}
\usepackage{graphicx,setspace}
\usepackage{amsfonts,amsmath, amssymb, amsthm}
\usepackage{mathrsfs}
\usepackage{epstopdf}
\usepackage{float}
\usepackage{lineno,hyperref}
\usepackage{color}
\usepackage{subfigure}
\usepackage{bm}

\usepackage{enumitem}
\modulolinenumbers[5]
\numberwithin{equation}{section}

\journal{?}

\begin{document}

\newtheorem{definition}{Definition}
\newtheorem{lemma}{Lemma}
\newtheorem{remark}{Remark}
\newtheorem{theorem}{Theorem}
\newtheorem{proposition}{Proposition}
\newtheorem{assumption}{Assumption}
\newtheorem{example}{Example}
\newtheorem{corollary}{Corollary}
\def\ep{\varepsilon}
\def\Rn{\mathbb{R}^{n}}
\def\Rm{\mathbb{R}^{m}}
\def\E{\mathbb{E}}
\def\hte{\hat\theta}
\newcommand{\eps}{\varepsilon}

\renewcommand{\b}{\beta}
\newcommand{\EX}{{\Bbb{E}}}
\newcommand{\PX}{{\Bbb{P}}}

\newcommand{\lam}{\lambda}
\newcommand{\s}{\sigma}
\renewcommand{\k}{\kappa}
\newcommand{\p}{\partial}
\newcommand{\D}{\Delta}
\newcommand{\om}{\omega}
\newcommand{\Om}{\Omega}
\renewcommand{\phi}{\varphi}
\renewcommand{\a}{\alpha}
\renewcommand{\b}{\beta}
\newcommand{\N}{{\mathbb N}}
\newcommand{\R}{{\mathbb R}}
\newcommand{\tr}{\triangle}

\newcommand{\de}{{\rm d}}
\newcommand{\dx}{{\rm d}x}
\newcommand{\dy}{{\rm d}y}
\newcommand{\dz}{{\rm d}z}
\newcommand{\dt}{{\rm d}t}
\newcommand{\delt}{\bigtriangleup t}

\newcommand{\B}{\mathcal B}
\newcommand{\Ro}{\mathbb R\setminus \{0\}}

\newcommand{\bess}{\begin{eqnarray*}}
\newcommand{\eess}{\end{eqnarray*}}

\newcommand{\befig}{\begin{figure}}
\newcommand{\enfig}{\end{figure}}

\newcommand{\bear}{\begin{eqnarray}}
\newcommand{\enar}{\end{eqnarray}}

\newtheorem{lem}{\textbf{\ \ \quad Lemma}}[section]
\newtheorem{col}{\textbf{\ \ \quad Corollary}}[section]
\newtheorem{prop}{\textbf{\ \ \quad Proposition}}[section]
\newtheorem{defi}{\textbf{\ \ \quad Definition}}[section]

\begin{frontmatter}

\title{{\bf Numerical analysis and applications of Fokker-Planck equations for stochastic dynamical systems with multiplicative $\alpha$-stable noises }}
\author{\centerline{\bf Yanjie Zhang$^{a,}\footnote{ zhangyj18@scut.edu.cn}$,
Xiao Wang$^{b,*}\footnote{corresponding author: xwang@vip.henu.edu.cn}$,
Qiao Huang$^{c,}\footnote{hq932309@alumni.hust.edu.cn}$,
Jinqiao Duan$^{d,}\footnote{duan@iit.edu}$
and
Tingting Li$^{b,}\footnote{ltt1120@mail.ustc.edu.cn}$}
\centerline{${}^a$ School of Mathematics,}
\centerline{South China University of Technology, Guangzhou 510641,  China}
\centerline{${}^b$ School of Mathematics and Statistics,} \centerline{Henan University, Kaifeng 475001, China}
\centerline{${}^c$ Center for Mathematical Sciences \& School of Mathematics and Statistics,}
\centerline{Huazhong University of Sciences and Technology, Wuhan 430074,  China}
\centerline{${}^d$ Department of Applied Mathematics,} \centerline{Illinois Institute of Technology, Chicago, IL 60616, USA}}

\begin{abstract}
The Fokker-Planck equations (FPEs) for stochastic systems driven by additive symmetric $\alpha$-stable noises may not adequately describe the time evolution for the probability densities of solution paths in some practical applications, such as hydrodynamical systems, porous media, and composite materials. As a continuation of previous works on additive case, the FPEs for stochastic dynamical systems with multiplicative symmetric $\alpha$-stable noises are derived by the adjoint operator method, which satisfy the nonlocal partial differential equations. A finite difference method for solving the nonlocal Fokker-Planck equation (FPE) is constructed, which is shown to satisfy the discrete maximum principle and to be convergent. Moreover, an example is given to illustrate this method. For asymmetric case,  general finite difference schemes are proposed, and some analyses of the corresponding numerical schemes are given. Furthermore, the corresponding result  is successfully applied  to the  nonlinear filtering problem.
\end{abstract}


\end{frontmatter}

%\linenumbers

\section{Introduction}

\par
The random fluctuations in nonlinear dynamical systems are usually non-Gaussian, such as in geosciences \cite{Dit}, biosciences \cite{Bli, Lin} and physical science \cite{Yong17, Ale} . There are experimental demonstrations of L\'evy fluctuations in optimal search theory and option pricing problem. Humphries et al. \cite{Hum} used maximum-likelihood methods to test for L\'evy patterns in relation to environmental gradients in the largest animal movement data set. They found L\'evy behaviour to be associated with less productive waters (sparser prey) and Brownian movements to be associated with productive shelf or convergence-front habitats (abundant prey). Jackson et al. \cite{Jac} presented a new efficient transform approach for regime-switching L\'evy models which was applicable to a wide class of path-dependent options and options on multiple assets, such as Bermudan, barrier, and shout options. Milovanov and Rasmussen \cite{Aj18} formulated the problem of confined L\'evy flight on a comb. They found that the L\'evy flight could be confined in the sense of generalized central limit theorem.  From that point of view, L\'evy noises are appropriate models for a class of important non-Gaussian processes with flights or bursts, which may be characterized by divergent moments.  When the dynamics become strongly intermediate, the $\alpha$-stable noises often occur.

A large number of investigations focused on the dynamical behaviors of stochastic differential equations subject to additive $\alpha$-stable  noises, including stochastic resonance \cite{Wan16}, coherence \cite{Zqw} and stochastic basins of attraction \cite{Ls}. Djeddi and Benidir \cite{Md} investigated the effect of additive impulsive noise modelled by $\alpha$-stable distributions on the Wigner-Ville Distribution and polynomial Wigner-Ville Distribution in the case of polynomial phase signals. Freitas et al. \cite{Ml} derived the capacity bounds for additive symmetric $\alpha$-stable noise channels. Mahmood et al. \cite{Am}  discussed and analyzed features of a good communications receiver for single-carrier modulation in additive impulsive noise. They showed that the conventional (continuous-time) receiver performed poorly in non-Gaussian additive white symmetric $\alpha$-stable noise in the It\^o sense.  However, in many practical applications, the additive $\alpha$-stable noise may not adequately describe the fluctuations of the stochastic force directly depend on the state of the system,  such as hydrodynamical systems, porous media, and composite materials. The stochastic descriptions of such systems must include a dependence on the process variable (the multiplicative noise) \cite{Ts10, Ts12, Ts13}.

Given the nonlinear stochastic dynamical systems which are driven by the multiplicative $\alpha$-stable noises, one of the main tasks is to quantify how uncertainly propagates and evolves. A popular method is to obtain the probability density function (PDF) of the solution paths, which contains the complete statistical information about the uncertainty and describes the change of probability of a stochastic process in space and time.  The PDF of the solution paths is governed by space fractional FPE. There are explicit formulae to obtain the associated fractional {FPE} for the stochastic differential equations (SDEs) excited by additive $\alpha$-stable noises \cite{Apple, Yxw19, Sune}. Similarly, for SDEs excited by the multiplicative $\alpha$-stable noises, Schertzer et al. \cite{ds} derived a fractional FPE by the methods of first and second characteristic functions. Sun et al. \cite{Sun17}  derived the FPE by the adjoint operator of the infinitesimal generators of Markov processes in the sense of Marcus. Moreover, they also implied  that obtaining the adjoint operators of the infinitesimal generators of It\^o SDEs with multiplicative $\alpha$-stable noises was still an open problem. Chechkin  et al. \cite{AV, AJ14,AJ18} derived the fractional FPE by the idea of transition probability in reciprocal space. Denisov et al. \cite{De09} derived the fractional FPE  by using a two-stage averaging procedure, which was associated with the Langevin equation.

For the fractional FPE, due to the existence of the nonlocal term, exact solutions can be obtained only for some special additive $\alpha$-stable noises with some restricted conditions \cite{Duan}. Therefore, some numerical schemes were developed to solve these types of equations. For example, Gao et al. \cite{Ting} proposed a fast and accurate finite difference method to simulate the nonlocal FPE on either a bounded or infinite domain. Meerschaert and Tadjeran  \cite{Mct} examined some finite numerical methods to solve a class of initial-boundary value fractional advection dispersion equation with variable coefficients on a finite domain. Huang and Oberman \cite{Huang} proposed a new numerical method based on the singular integral representation for the operator. The method combined finite differences with numerical quadrature to obtain a discrete convolution operator with positive weights. Liu et al. \cite{Liu} transformed the space fractional FPE into an ordinary differential equation, which could be solved through backward difference formula.  Yang et al. \cite{Yft} considered the numerical solution of a fractional partial differential equation with Riesz space fractional derivatives on a finite domain. Ren et al. \cite{Rhw} proposed an efficient method based on the shifted Chebyshev-tau idea to solve an initial-boundary value problem for the fractional diffusion equations.  Xu et al. \cite{Yxu} used the path integral method to solve one-dimensional space fractional FPE.

Much effort has been made on the applications of FPEs for stochastic dynamical systems with additive $\alpha$-stable noises, one of the important applications is to the nonlinear filtering problem, which key ingredient is the state transition density \cite{spss, Hjk12}. Recently, we have successfully applied  FPE to the  nonlinear filtering problem \cite{zhang, qiao}. As a continuation of previous works on additive case, in this paper, our first goal is to derive the FPEs for It\^o SDEs with multiplicative $\alpha$- stable  noise by the adjoint operator method. The difficulty lies in obtaining the explicit expressions for the adjoint operators of the infinitesimal generators associated with such SDEs. The second is to develop an accurate numerical scheme with stability and convergence analysis for one-dimensional case, on this basis, to extend these results to asymmetric case. The third is to apply FPE to the nonlinear filtering problem.

The rest of this paper is organized as follows. After some preliminaries in Section 2, we derive the nonlocal FPE for SDE  with multiplicative symmetric $\alpha$-stable noise in Section 3. An accurate numerical scheme is proposed to simulate the nonlocal FPE in bounded and unbounded domain. Moreover, some examples are given to illustrate our main results. In section 4, we extend the above results to the asymmetric case. In Section 5, we apply the above results to the nonlinear filtering problem. The paper is concluded in Section 6.
%\par
% In this paper, C denotes the generic constant which may vary from line to line.
\section{Preliminaries } We recall some basic facts about $\alpha$-stable noise.
\label{sec:pre}

\begin{definition}
A real-valued stochastic process $L=\{L_t\}_{t\ge0}$ is called an one-dimensional L\'evy  process if
\begin{enumerate}
\item[(i)]$L_0$=0, ~(a.s);
\item[(ii)]$L_t$ has independent increments and stationary increments;
\item[(iii)]$L_t$ has stochastically continuous sample paths.
\end{enumerate}
\end{definition}

An one-dimensional L\'evy process $L$ is characterized by a drift coefficient $b \in {\mathbb{R}}$, a non-negative constant $ Q $ and a Borel measure $\nu$ defined on ${\mathbb{R}}\backslash \{ 0\} $. And we call $(b,Q,\nu)$ the generating triplet of the L\'evy process $L$. Moreover, we have the L\'evy-It\^o decomposition for $L$ as follows:

\begin{equation*}
{L_t}= bt + \sqrt{Q}B_t + \int_{\left| y \right| < 1} y\widetilde N(t,dy) + \int_{\left| y \right| \ge 1} y N(t,dy),
\end{equation*}
where $N:[0,\infty)\times \B(\Ro)\to\N$ is the Poisson random measure with jump measure $\nu$, i.e., $\nu (S) = \mathbb{E}N(1,S)$ for all $S\in\B(\Ro)$, the measure $\nu $ is called the L\'evy measure satisfying $\int_{\Ro}(x^2\wedge1)\nu(dx) <\infty$, $\widetilde N(dt,dy) = N(dt,dy) - \nu (dy)dt$ is the compensated Poisson random measure, and $B_t$ is a standard one-dimensional Brownian motion independent with $N$.

The characteristic function of $L$ is given by
\begin{equation*}
\mathbb{E}[\exp({\rm i} u L_t )]=\exp(t\psi(u)), ~~~u \in {\mathbb{R}},
\end{equation*}
where the function $\psi:{\mathbb{R}}\rightarrow \mathbb{C}$ is called the characteristic exponent of $L$ and has the following representation
\begin{equation*}
\psi(u)={\rm i} u b-\frac{Q}{2} u^2 +\int_{{\mathbb{R}}\backslash \{ 0\}}{\left(e^{{\rm i}uz }-1-{\rm i} uz {I_{\{ \left| z \right| < 1\} }}\right)\nu(dz)}.
\end{equation*}

\begin{definition}
For $\alpha \in (0,2)$, an one-dimensional symmetric $\alpha$-stable process $L^{\alpha}=\{L^{\alpha}_{t}\}_{t\ge0}$ is a L\'evy process such that its characteristic exponent $\psi$ is given by
\begin{equation*}
\psi(u)=-\lvert u \rvert^{\alpha},  ~for~u \in {\mathbb{R}}.
\end{equation*}

\end{definition}
Therefore, for one-dimensional symmetric $\alpha$-stable  process, the diffusion matrix $ Q= 0$,
the drift vector $ b= 0$, and the L\'evy measure is given by
\begin{equation}\label{nu}
\nu(du)=\frac{c(1,\alpha)}{{\lvert u \rvert}^{1+\alpha}}du=:\nu^{\alpha}(du),
\end{equation}
where $ c(1, \alpha):=\alpha\Gamma((1+\alpha)/2)/{(2^{1-\alpha}\pi^{1/2}\Gamma(1-\alpha/2))}$ and $\Gamma$  is the Gamma function.

For every  measurable function $\phi\in \mathcal{C}^2_{0}({\mathbb{R}})$, the generator for the one-dimensional symmetric $\alpha $-stable
process is
\begin{equation*}
A\phi(x)={\rm{{\mathbf{P.V}.}}} \int_{{\mathbb{R}}\backslash \{ 0\}}\left[\phi(x+u)-\phi(x)\right]\nu^{\alpha}(du),
\end{equation*}
where ${\rm{{\mathbf{P.V}.}}}$ stands for the Cauchy principle value; that is
\begin{equation*}
A\phi(x)=\lim_{\varepsilon\rightarrow 0} \int_{\{u\in \mathbb{R}: |x-u|>\varepsilon\}}\left[\phi(x+u)-\phi(x)\right]\nu^{\alpha}(du).
\end{equation*}
Here the Fourier transform for $\phi$ is defined by
\begin{equation}
\mathbb{F}(\phi)(k)=\frac{1}{\sqrt{2\pi}}\int_{\mathbb{R}}e^{-ikx}\phi(x)dx.
\end{equation}

It is known in \cite{Apple} that ${A}$ extends uniquely to a self-adjoint operator in the domain. By Fourier inverse transform, we have
\begin{equation*}
A\phi=(-\Delta)^{\alpha/2}\phi.
\end{equation*}

\section{FPE for symmetric case}
 \par
Let us consider the following stochastic differential equation
\begin{equation}
\label{ori}
 d{X_t} = {f}({X_t})dt + {g}({X_t})d{B_t} + \sigma({X_{t-}})dL_t^{{\alpha}},\quad X_0=x_0,
\end{equation}
where $x_0\in \R$, $f$ is  a given deterministic vector field, ${g}$ is  the diffusion coefficient, $\sigma$ is called the noise intensity, $X_{t-}$ means
the left limit, i.e.,  $X_{t-}=\lim_{l \uparrow t} X(l)$, $B=\{B_t\}_{t\ge 0}$ is a standard Brownian motion, $L^{\alpha}=\{L^\alpha\}_{t\ge 0}$ is an one-dimensional symmetric $\alpha$-stable  process with generating triplet $(0,0,{\nu}^{\alpha})$,  which is independent of $B$.
\par
We make the following assumptions  on the drift and  diffusion coefficients.
\par
{\bf Hypothesis  H1. }
The nonlinear terms $f, g$ satisfy the  Lipschitz conditions, i.e., there exists a positive constant $\gamma$ such that for all $x_1,x_2\in \mathbb{R}$,
 \begin{equation*}
 \begin{aligned}
 &\ \lvert f(x_1)-f(x_2)\rvert^{2}+\lvert g(x_1)-g(x_2)\rvert^{2}+\int_{|y|<1}|\sigma(x_1)y-\sigma(x_2)y|^{2}\nu^{\alpha}(dy) \\
 \le&\ \gamma \left[\lvert x_1-x_2\rvert^2\right].
 \end{aligned}
 \end{equation*}

\par
{\bf Hypothesis  H2. }  There exists $K>0$ such that for all $x\in \mathbb{R}$,
\begin{equation*}
2xf(x)+|g(x)|^{2}+\int_{|y|<1}|\sigma(x)y|^{2}\nu^{\alpha}(dy)\leq K(1+x^{2}).
\end{equation*}

\par
{\bf Hypothesis  H3. } The function $|\sigma(\cdot)|^{\alpha}\in {C^{1}(\mathbb{R})}$ and $\sigma(x)\neq 0$ for all $x$.
\begin{remark}
Under Hypotheses $\bf{H1}$ and $\bf{H2}$, the stochastic differential equation \eqref{ori} has a unique global solution \cite[Theorem 3.1]{Wu}.
\end{remark}
\subsection{Derivation of FPE for symmetric case}
By the It\^o-L\'evy decomposition theorem, the SDE \eqref{ori} can be rewritten as
\begin{equation}
\label{orio}
d{X_t} = {f}({X_t})dt + {g}({X_t})d{B_t} + \int_{\left| y \right| < 1} \sigma({X_{t-}}) y \widetilde N(dt,dy)+
\int_{\left| y \right| \geq 1} \sigma({X_{t-}}) y N(dt,dy),~~~ X_0=x_0,
\end{equation}
Using the similar method of \cite[Theorem 6.7.4]{Apple}, the generator for the SDE \eqref{orio} in the case of symmetric $\alpha$-stable process is
\begin{equation*}
\begin{aligned}
A\varphi(x)=&\ f(x)\varphi^{'}(x)+\frac{1}{2}g^{2}(x)\varphi^{''}(x) \\
&\ +\int_{{\mathbb{R}}\backslash \{ 0\}}\left[\varphi\left(x+\sigma(x)y \right)-\varphi(x)-\sigma(x)y\varphi^{'}(x){I}_{\{|y|<1\}}(y)\right]\nu^{\alpha}(dy).
\end{aligned}
\end{equation*}
Under Hypothesis $\bf{H3}$, we have
\begin{equation*}
\begin{aligned}
A\varphi(x)=&\ f(x)\varphi^{'}(x)+\frac{1}{2}g^{2}(x)\varphi^{''}(x)+\int_{{\mathbb{R}}\backslash \{ 0\}}\left[\varphi\left(x+\sigma(x)y \right)-\varphi(x)-\sigma(x)y\varphi^{'}(x){I}_{\{|y|<1\}}(y)\right]\nu^{\alpha}(dy)\\
=&\ f(x)\varphi^{'}(x)+\frac{1}{2}g^{2}(x)\varphi^{''}(x)+|\sigma(x)|^{\alpha}\int_{{\mathbb{R}}\backslash \{ 0\}}\left[\varphi\left(x+z\right)-\varphi(x)-z\varphi^{'}(x){I}_{\{|z|<1\}}(z)\right]\nu^{\alpha}(dz)\\
&\ +|\sigma(x)|^{\alpha}\int_{{\mathbb{R}}\backslash \{ 0\}}z\varphi^{'}(x)\left[{I}_{\{|z|<1\}}(z)-{I}_{\{|z|<|\sigma(x)|\}}(z)\right]\nu^{\alpha}(dz)\\
=:&\ I_1+I_2+I_3+I_4.
\end{aligned}
\end{equation*}
\begin{remark}
Since $\nu^\alpha$ in \eqref{nu} is symmetric, we have $I_4=0$ in the sense of the Cauchy principle value.
\end{remark}
Let us try to find the adjoint operator $A^{*}$ in the Hilbert space $ L^{2}(\mathbb{R})$. The adjoint parts for the
first two terms $I_1$ and $I_2$ in $ A$ are easy to find via integration by parts: $-(f(x)\varphi(x))^{'}$ and $\frac{1}{2}(g^{2}(x)\varphi)^{''}$. For the third term $I_3$, denoted by $\tilde{{A}}\varphi$, we have the  following lemma.
\begin{lemma}
Under Hypotheses $\bf{H1}$-$\bf{H3}$, the adjoint operator of $\tilde{{A}}$ in the sense of the Cauchy principle value is
\begin{equation*}
\begin{aligned}
\tilde{A}^{*} v(x)&=\int_{{\mathbb{R}}\backslash \{ 0\}}\left[|\sigma(x+z)|^{\alpha}v(x+z)-|\sigma(x)|^{\alpha}v(x)\right]\nu^{\alpha}(dz).
\end{aligned}
\end{equation*}
\end{lemma}

\begin{proof}
By the definition of adjoint operator, for $v$ in the domain of $\tilde{A}^{*}$, we have
\begin{equation*}
\begin{aligned}
&\ \int_{\mathbb{R}}\tilde{{A}}\varphi(x)v(x)dx \\
=&\ \int_{\mathbb{R}}|\sigma(x)|^{\alpha}\int_{{\mathbb{R}}\backslash \{ 0\}}\left[\varphi\left(x+z\right)-\varphi(x)-z\varphi^{'}(x){I}_{\{|z|<1\}}(z)\right]\nu^{\alpha}(dz)v(x)dx \\
=&\ \int_{{\mathbb{R}}\backslash \{ 0\}}\left\{\int_{\mathbb{R}}\left[\varphi\left(x+z\right)-\varphi(x)-z\varphi^{'}(x){I}_{\{|z|<1\}}(z)\right]|\sigma(x)|^{\alpha}v(x)dx\right\}\nu^{\alpha}(dz)\\
=&\ -\int_{\mathbb{R}}\varphi(x)\left\{\int_{{\mathbb{R}}\backslash \{ 0\}}\left[|\sigma(x)|^{\alpha}v(x)-|\sigma(x-z)|^{\alpha}v(x-z)-{I}_{\{|z|<1\}}(z)z\left(|\sigma(x)|^{\alpha}v(x)\right)^{'}\right]\nu^{\alpha}(dz)\right\}dx.
\end{aligned}
\end{equation*}
Therefore, the adjoint operator of $\tilde{{A}}$ is
\begin{equation*}
\begin{aligned}
\tilde{A}^{*}v(x)&=-\int_{{\mathbb{R}}\backslash \{ 0\}}\left[|\sigma(x)|^{\alpha}v(x)-|\sigma(x-z)|^{\alpha}v(x-z)-{I}_{\{|z|<1\}}(z)z\left(|\sigma(x)|^{\alpha}v(x)\right)^{'}\right]\nu^{\alpha}(dz)\\
&={\rm{{\mathbf{P.V}.}}}\int_{{\mathbb{R}}\backslash \{ 0\}}\left[|\sigma(x+z)|^{\alpha}v(x+z)-|\sigma(x)|^{\alpha}v(x)\right]\nu^{\alpha}(dz).
\end{aligned}
\end{equation*}
\end{proof}

By Lemma 1, we obtain the following FPE in the case of multiplicative  $\alpha$-stable L\'evy noise.
\begin{theorem}
Under Hypotheses $\bf{H1}$-$\bf{H3}$, the Fokker-Planck equation for SDE \eqref{ori} is
\begin{equation}
\label{density0}
\frac{\partial p}{\partial t}=A^{*}p, ~~~~p(x,0)=\delta(x-x_0),
\end{equation}
where
\begin{equation*}
A^{*}p(x,t)=-\partial_{x}\left(f(x)p(x,t)\right)+\frac{1}{2}\partial_{xx}\left(g^{2}(x)p(x,t)\right)+\tilde{A}^{*}p(x,t).
\end{equation*}
\end{theorem}
\begin{remark}
In case of $g\equiv0$, the existence of the solution of \eqref{density0} can be found in \cite[Theorem 2.2]{Wang}.
\end{remark}
\subsection{Numerical method}
In this subsection, we need to make a modification for Hypothesis $\bf{H3}$. Precisely, we require $|\sigma(\cdot)|^{\alpha}\in {C^{2}(\mathbb{R})}$ and  $\sigma(x)\neq 0$ for all $x\in \mathbb{R}$.
First, we present the numerical schemes for the absorbing boundary condition, that is, $p(x,t)=0$ for all $t\ge0$, when
$x \notin (-1,1)$.

Set $u(x,t)=|\s(x)|^{\a}p(x,t)$, in the sense of Cauchy principle value, we have
\begin{equation}
\begin{aligned}
\label{density1}
u_t(x,t) =&\ -|\s(x)|^{\a}\partial_{x}\left(\frac{f(x)}{|\s(x)|^\a}u(x,t)\right)+\frac{1}{2}|\s(x)|^{\a}\partial_{xx}\left(\frac{g^{2}(x)}{|\s(x)|^\a}u(x,t)\right)  \\
&+|\s(x)|^\a \int_{{\mathbb{R}}\backslash \{ 0\}}\bigg[u(x+z,t)- u(x,t)\bigg]\nu^{\alpha}(dz) \\
=&\ \frac{1}{2}g^{2}(x)\partial_{xx}\left(u(x,t)\right)+M(x)\partial_{x}(u(x,t))(x,t)+N(x)u(x,t)\\
&+|\s(x)|^\a\int_{{\mathbb{R}}\backslash \{ 0\}}\bigg[u(x+z,t)- u(x,t)\bigg]\nu^{\alpha}(dz),
\end{aligned}
\end{equation}
where
\begin{equation*}
\left\{
\begin{aligned}
M(x)&=|\s(x)|^{\a}\left(\frac{g^{2}(x)}{|\s(x)|^{\a}}\right)^{'}-f(x), \\
N(x)&=\frac{1}{2}|\s(x)|^\a\left(\frac{g^{2}(x)}{|\s(x)|^{\a}}\right)^{''}-|\s(x)|^{\a}\left(\frac{f(x)}{|\s(x)|^{\a}}\right)^{'}.
\end{aligned}
\right.
\end{equation*}

Due to the absorbing boundary condition, the above equation~(\ref{density1}) becomes

\bess
    u_t(x,t) &=& \frac{1}{2}g^{2}(x)\partial_{xx}(u(x,t)) + M(x) \partial_{x}(u(x,t)) + \tilde{N}(x)u(x,t) \nonumber\\
        &&+ c(1,\alpha) |\s(x)|^\a \int_{-1-x}^{1-x} \frac{u(x+z,t)-u(x,t)}{|z|^{1+\a}} \dz,
\eess
where
\bess
    \tilde{N}(x) = N(x)- \frac{c(1,\alpha)|\s(x)|^\a}{\a} \left[ \frac{1}{(1+x)^\a}+ \frac{1}{(1-x)^\a} \right].
\eess

Let us divide the interval $[-2,2]$ into $4J$ subintervals and define $x_j=jh$ for $-2J \leq j \leq 2J$,
where $h=\frac{1}{J}$. We denote $U_j$ as the numerical solution for $u$ at $(x_j, t)$. For
the first derivative, we use the upwind scheme as follows,
\bess
     \delta_u U_j=
    \begin{cases}
       \frac{U_{j}-U_{j-1}}{h}\; ,   & \text{ if $ M(x_j) < 0 $,}\\
       \frac{U_{j+1}-U_{j}}{h}\; ,   & \text{ if $ M(x_j) > 0 $}.
    \end{cases}
\eess

Using a modified trapezoidal rule for the singular integral, we have the following semi-discrete scheme

\bear
\label{U}
    \frac{\de U_j}{\de t} &= C_h \frac{U_{j-1}-2U_j+U_{j+1}}{h^2}  + M(x_j) \delta_u U_j
     +\tilde{N}(x_j)U_j \nonumber \\
     &+ c(1,\a) h|\s(x_j)|^\a \sum\limits_{k=-J-j, k\neq 0}^{J-j} \frac{U_{j+k}-U_j}{|x_k|^{1+\a}} ,
\enar
where the summation symbol $\sum$ means the terms of both end indices are multiplied by $\frac12$, and
\bess
    C_h= \frac{g(x_j)^2}{2} - c(1,\a) \zeta(\a-1)|\s(x_j)|^\a h^{2-\a},
\eess
where $\zeta$ is the Riemann zeta function.

For the natural condition, the semi-discrete scheme  becomes
\begin{equation}
\begin{aligned}
\label{nature0}
    \frac{\de U_j}{\de t} &=& C_h \frac{U_{j-1}-2U_j+U_{j+1}}{h^2}  + M(x_j) \delta_u U_j
     + N(x_j)U_j \\
     &&+ c(1,\a) h|\s(x_j)|^\a \sum\limits_{k=-J-j, k\neq 0}^{J-j} \frac{U_{j+k}-U_j}{|x_k|^{1+\a}} ,
\end{aligned}
\end{equation}
where $J=\tilde{L}/h$ and $\tilde{L} \gg 1$.
\begin{remark}
The domain $(-\tilde{L}, \tilde{L})$ is required to be large enough so as to make the semi-discrete scheme convergence.
\end{remark}

\begin{prop}
[Maximum principle for the absorbing condition]
 For the absorbing boundary condition and explicit Euler for time derivative, the scheme \eqref{U}
 satisfies the discrete maximum principle with $f=g=0$, if $|\s(x)|\leq \widetilde{M}$($\widetilde{M}$ is a constant),
 and the $\bigtriangleup t$ and $h$ satisfy the following condition,
 \bear
 \label{MPcondition0}
 \frac{\bigtriangleup t }{h^\a}  \leq \frac{1}{2 \widetilde{M}^\a c(1,\a) [1+\frac{1}{\a}-\zeta(\a-1)] }.
 \enar
\end{prop}

\begin{proof}
See Appendix A.1.
\end{proof}

Similarly, we will gain the maximum principle for the natural condition.

\begin{prop}[Maximum principle for the nature condition] \label{propMax1}
 For the nature condition and explicit Euler for time derivative, the scheme (\ref{nature0})
 satisfies the discrete maximum principle with $f=g=0$, if the $\bigtriangleup t$ and $h$ satisfy the following condition,
 \bear  \label{MPcondition}
   \frac{\bigtriangleup t }{h^\a}  \leq \frac{1}{2 \widetilde{M}^\a c(1,\a) [1+\frac{1}{\a}-\zeta(\a-1)] }.
 \enar

\end{prop}

By the explicit Euler method for time integration \eqref{nature0}, then we have
\begin{equation}
\label{nature1}
\begin{aligned}
U_j^{n+1} =&\ U_j^n-\delt  c(x_j) \zeta(\a-1)h^{-\a} (U_{j+1}^n-2U_j^n+U_{j-1}^n) \\
            &\ - \frac{c(x_j)\delt}{\a}  \bigg[\frac{1}{(1+x_j)^\a} +\frac{1}{(1-x_j)^\a} \bigg]U_j^n
            + c(x_j)\delt h \sum\limits_{k=-J-j, k\neq 0}^{J-j} \frac{U_{j+k}^n-U_j^n}{|x_k|^{1+\a}}, \\
\end{aligned}
\end{equation}
where $J=\tilde{L}/h$ and $\tilde{L} \gg 1$.

When the ratio ${\bigtriangleup t }/{h^\a}$ satisfies the condition \eqref{MPcondition}, the explicit scheme \eqref{nature1} is  stable by the the linearity and discrete maximum principle.

In the following, we will present  the convergence analysis for the natural far-field condition.
\begin{prop}
The numerical solution $U^{n}_{j}$ of \eqref{nature1} converges to the analytic solution to \eqref{density1} for $x_j$ in $[-\tilde{L}/2, \tilde{L}/2]$ when the refinement path satisfies \eqref{MPcondition} and the length of the integration interval $2\tilde{L}$  in \eqref{nature1} tends to $\infty$.
\end{prop}
\begin{proof}
See Appendix A.2.
\end{proof}

\subsection{Numerical experiments}

\begin{example}
Consider the Langevin equation is of the form
\begin{equation}
\label{example2}
dX_t=f(X_t)dt+g(X_t)d B_t+\sigma(X_t)dL^{\alpha}_{t}, ~~~~X_0=0,
\end{equation}
where the function $f$ is determined by a function $V$, so that $ f=-\partial V(x,t)/\partial x$. The function $V$ could be called a climatic pseudo-potential in geosciences, energy potential in physics or profit or cost function in economics and optimization. Here we consider a motion in the time-independent bistable potential
\begin{equation}
V(x)=0.1x^2.
\end{equation}
The first noise term  is a Gaussian noise with intensity $g(x)$. The second noise term is an $\alpha$-stable noise with intensity $\sigma(x)=2+\sin(x)$. It is noting that we choose periodically modulated noise here. Moreover, the above model may be applied to explain the periodic recurrence of the earth's ice age on Earth \cite{yu}.
\end{example}

\befig[h]
\begin{center}
\includegraphics*[width=0.7\linewidth]{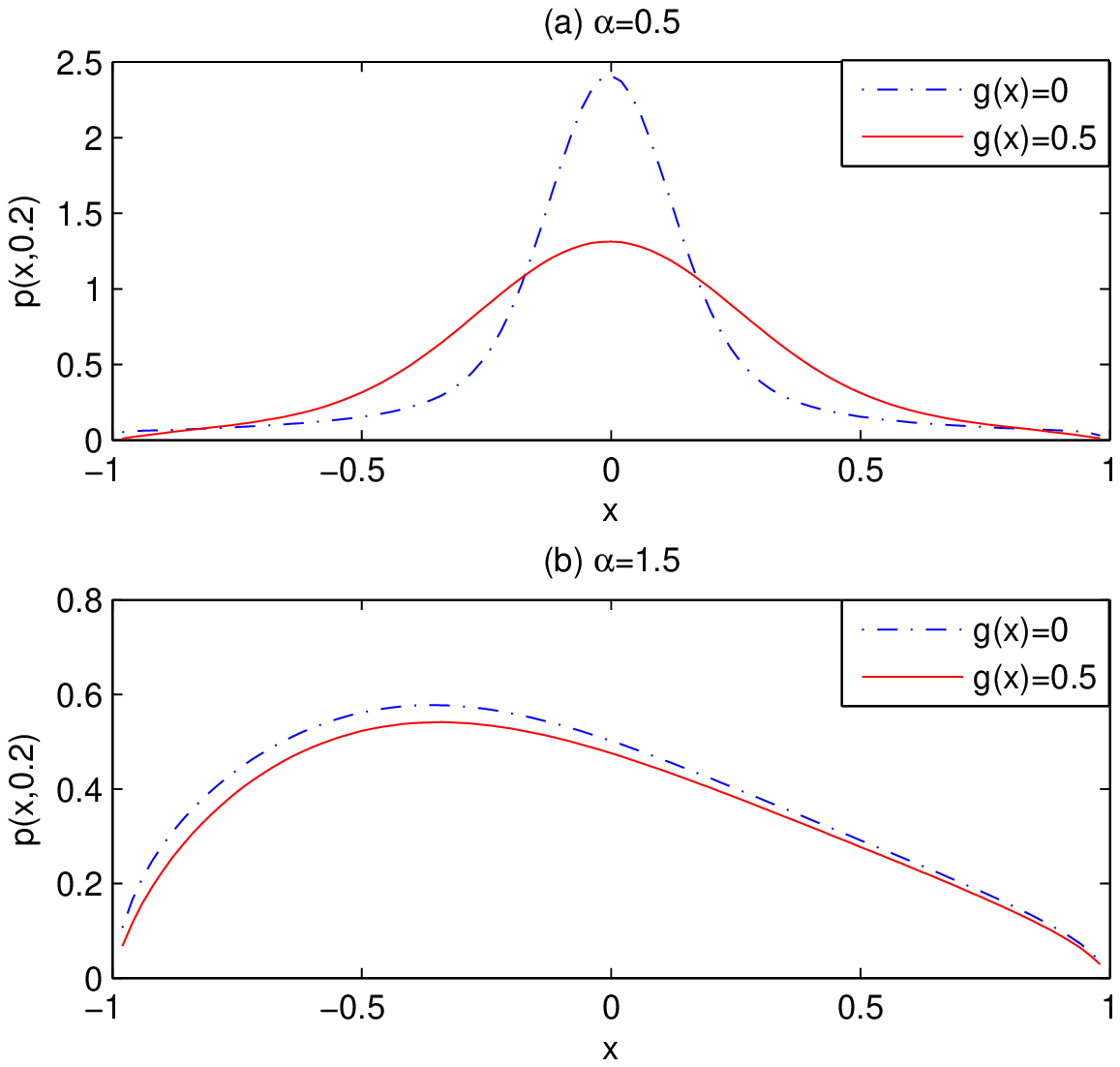}
\end{center}
\caption{The FPE driven by multiplicative $\alpha$-stable L\'evy motions with $f(x)=-0.2x$
and $ D=(-1,1), t=0.2$ for different $\alpha=0.5, 1.5$ and $d=0, 0.5$. The initial condition is $p(x,0.01) = \sqrt{\frac{40}{\pi}}e^{-40x^2}$. }
\label{multFPEdifd}
\enfig
 Here, we examine the FPE driven by multiplicative  $\alpha$-stable L\'evy motions with absorbing condition. We take the initial condition $p(x,0.01) = \sqrt{\frac{40}{\pi}}e^{-40x^2}$ and $ D=(-1,1) $ at time $t=0.2$. Fig.~\ref{multFPEdifd} shows the solutions of  Eq.~(\ref{density0}) for different $\alpha$ and $g(x)$. It suggests that, the Gaussian noise intensity is larger, the solution becomes smaller near the origin for $\a=0.5$, while it is larger far away from the origin.
For $\a=1.5$, the solution is  smaller as the Gaussian noise intensity becomes larger.
It seems the Gaussian noise have greater effect than the  multiplicative L\'evy noise for $\a=0.5$.
But, it is opposite for $\a=1.5$.

\section{Extending FPE to asymmetric case}
%\subsection{Framework}
\setcounter{equation}{0}
 In this section, we consider the SDEs driven by multiplicative asymmetric one-dimensional L\'evy motions and derive the corresponding FPEs. Moreover, an efficient numerical scheme for the probability density function is developed. For an asymmetric $\a$-stable one-dimensional L\'evy motion, the L\'evy measure $\nu_{\a,\beta}$ is given in \cite{Xiao2019} as follow,
\bess
   \nu_{\a,\beta}({\rm d}y) = \frac{C_p(\beta) 1_{\{0<y<\infty\}}(y)+C_n(\beta) 1_{\{-\infty<y<0\}}(y)}{|y|^{1+\alpha}} {\rm d}y,
\eess
where the parameter $\beta \in (-1,1)$ represents the non-symmetry of  $\nu_{\a,\beta}$, and
\bess
     C_p(\beta) = C_{\a}\frac{1+\beta}2, \quad C_n(\beta)= C_{\a}\frac{1-\beta}2,
\eess
with
\bess
    C_{\a} =
    \begin{cases}
        \frac{\a(1-\a)}{\Gamma(2-\a)\cos{(\frac{\pi \a}2)}}\;,   &\text{ $ \a \neq 1 $,}\\
        \frac2{\pi},  \,  &\text{ $ \a = 1$.}
    \end{cases}
\eess
Here we introduce the notation $f(x)\uparrow_{+} x$, which represents the function $f(x)$ is positive and strictly monotonically increasing in $x$. Similarly, the notation $f(x)\downarrow_{-}x$ represents the function $f(x)$ is negative and strictly monotonically decreasing in $x$.
\subsection{Derivation of FPE for asymmetric case}

Let us consider the following stochastic differential equation
\begin{equation}
\label{ori}
 d{X_t} = {f}({X_t})dt + {g}({X_t})d{B_t} + \sigma({X_{t-}})dL_t^{{\alpha, \beta}},\quad X_0=x_0.
\end{equation}
Then the generator for the SDE \eqref{ori} driven by asymmetric $\alpha$-stable L\'evy motion is
\begin{equation*}
\begin{aligned}
{B}\varphi(x) &=f(x)\varphi^{'}(x)+\frac{1}{2}g^{2}(x)\varphi^{''}(x)+\int_{{\mathbb{R}}\backslash \{ 0\}}\left[\varphi\left(x+\sigma(x)y \right)-\varphi(x)-\sigma(x)y\varphi^{'}(x){I}_{\{|y|<1\}}(y)\right]\nu_{\alpha,\beta}(dy)\\
&:=K_1+K_2+K_3.
\end{aligned}
\end{equation*}
In the following, we aim at finding the adjoint operator $B^{*}$ in the Hilbert space $ L^{2}(\mathbb{R})$.
Here we need an additional hypothesis.
\par
{\bf Hypothesis \bf{H4}. }(Condition for noise intensity) The function $\sigma$ is strict monotone.
\begin{remark}
In fact, Hypothesis $\bf{H4}$ can be weakened. We only need  the  inverse function of $\sigma(x)$  exists.
\end{remark}
Also, the adjoint parts for the
first two terms $K_1$ and $K_2$ in $ B$ are easy to find via integration by parts: $-\left(f(x)\varphi(x)\right)^{'}$ and $\frac{1}{2}\left(g^{2}(x)\varphi\right)^{''}$. For the third term $K_3$, we denote it by $\widetilde{{Q}}\varphi$, i.e.,
\begin{equation*}
\widetilde{{Q}}\varphi(x)=\int_{{\mathbb{R}}\backslash \{ 0\}}\left[\varphi\left(x+\sigma(x)y \right)-\varphi(x)-\sigma(x)y\varphi^{'}(x){I}_{\{|y|<1\}} (y)\right]\nu_{\alpha,\beta}(dy).
\end{equation*}
We have the following lemma.
\begin{lemma}
Under Hypotheses $\bf{H1}$-$\bf{H4}$, the adjoint operator of $\widetilde{{Q}}$ is
\begin{equation*}
\begin{aligned}
\widetilde{Q}^{*}v(x)=&\ \int_{{\mathbb{R}}\backslash \{ 0\}}\left[|\sigma(x+z)|^{\alpha}v(x+z)-|\sigma(x)|^{\alpha}v(x)-\left[|\sigma(x)|^{\alpha}v(x)\right]^{'}z{I}_{\{|z|<|\sigma(x)|\}}(z)\right]\tilde{\nu}_{\alpha,\beta}(dz)\\
&\ +\beta C_{\alpha}v(x)\sigma^{'}(x),
\end{aligned}
\end{equation*}
where
$$ \tilde{\nu}_{\alpha,\beta}(dz)=
\begin{cases}
\nu_{\alpha,-\beta}(dz), &\text{if }\sigma(x)\uparrow_{+} x, \\
\nu_{\alpha,\beta}(dz),  &\text{if }\sigma(x)\downarrow_{-} x. \\
\end{cases}
$$
\end{lemma}
\begin{proof}
We assume $\sigma(x)\uparrow_{+} x$. By the definition of adjoint operator and the transformation $z=\s(x) y$, we have
\begin{equation*}
\begin{aligned}
&\int_{\mathbb{R}}\left\{\int_{{\mathbb{R}}\backslash \{ 0\}}\left[\varphi\left(x+\sigma(x)y\right)-\varphi(x)-\sigma(x)y\varphi^{'}(x){I}_{\{|y|<1\}}(y)\right]\nu_{\alpha,\beta}(dy)\right\}v(x)dx \\
&=\int_{\mathbb{R}}\left\{\int_{{\mathbb{R}}\backslash \{ 0\}}\left[\varphi\left(x+z\right)-\varphi(x)-z\varphi^{'}(x){I}_{\{|z|<\sigma(x)\}}(z)\right]\nu_{\alpha,\beta}(dz)\right\}(\sigma(x))^{\alpha}v(x)dx \\
&=\int_{{\mathbb{R}}\backslash \{ 0\}}\left\{\int_{\mathbb{R}}\left[\varphi\left(x+z\right)-\varphi(x)-z\varphi^{'}(x){I}_{\{|z|<\sigma(x)\}}(z)\right](\sigma(x))^{\alpha}v(x)dx\right\}\nu_{\alpha,\beta}(dz) \\
&=\int_{{\mathbb{R}}\backslash \{ 0\}}\left\{\int_{-\infty}^{\infty}\varphi\left(x\right)(\sigma(x-z))^{\alpha}v(x-z)dx-\int_{-\infty}^{\infty}\varphi(x)(\sigma(x))^{\alpha}v(x)dx \right.\\
&\phantom{=\;\;}\left.-z\int_{-\infty}^{\infty}\varphi^{'}(x){I}_{\{|z|<\sigma(x)\}}(z)(\sigma(x))^{\alpha}v(x)dx\right\}\nu_{\alpha,\beta}(dz)\\
&=\int_{{\mathbb{R}}\backslash \{ 0\}}\left\{\int_{-\infty}^{\infty}\varphi\left(x\right)(\sigma(x-z))^{\alpha}v(x-z)dx-\int_{-\infty}^{\infty}\varphi(x)(\sigma(x))^{\alpha}v(x)dx \right.\\ &\phantom{=\;\;}\left.+\int_{-\infty}^{\infty}z\varphi(x){I}_{\{|z|<\sigma(x)\}}(z)\left[(\sigma(x))^{\alpha}v(x)\right]^{'}dx\right\}\nu_{\alpha,\beta}(dz)\\
&+\int_{{\mathbb{R}}\backslash \{ 0\}}z|z|^{\alpha}\varphi\left(\sigma^{-1}(|z|)\right)v\left(\sigma^{-1}(|z|)\right)\nu_{\alpha,\beta}(dz)\\
&=\left\langle \varphi(x), \int_{{\mathbb{R}}\backslash \{ 0\}}\left[(\sigma(x-z))^{\alpha}v(x-z)-(\sigma(x))^{\alpha}v(x)+\left[(\sigma(x))^{\alpha}v(x)\right]^{'}z{I}_{\{|z|<\sigma(x)\}}(z)\right]\nu_{\alpha,\beta}(dz)\right\rangle \\
&+\beta C_{\alpha}\left\langle \varphi(x), v(x)\sigma^{'}(x) \right\rangle,
\end{aligned}
\end{equation*}
then the adjoint operator of $\widetilde{Q}$ is
\begin{equation}
\label{q1}
\begin{aligned}
&\widetilde{Q}^{*}v(x)=\beta C_{\alpha}v(x)\sigma^{'}(x)\\
&+\int_{{\mathbb{R}}\backslash \{ 0\}}\left[(\sigma(x-z))^{\alpha}v(x-z)-(\sigma(x))^{\alpha}v(x)+\left[(\sigma(x))^{\alpha}v(x)\right]^{'}z{I}_{\{|z|<\sigma(x)\}}(z)\right]\nu_{\alpha,\beta}(dz)\\
&=\beta C_{\alpha}v(x)\sigma^{'}(x)\\
&+\int_{{\mathbb{R}}\backslash \{ 0\}}\left[(\sigma(x+z))^{\alpha}v(x+z)-(\sigma(x))^{\alpha}v(x)-\left[(\sigma(x))^{\alpha}v(x)\right]^{'}z{I}_{\{|z|<\sigma(x)\}}(z)\right]\nu_{\alpha,-\beta}(dz).
\end{aligned}
\end{equation}
Similarly, when $\sigma(x)\downarrow_{-} x$ , we have
\begin{equation*}
\begin{aligned}
&\int_{\mathbb{R}}\left\{\int_{{\mathbb{R}}\backslash \{ 0\}}\left[\varphi\left(x+\sigma(x)y\right)-\varphi(x)-\sigma(x)y\varphi^{'}(x){I}_{\{|y|<1\}}(y)\right]\nu_{\alpha,\beta}(dy)\right\}v(x)dx \\
&=\int_{-\infty}^{\infty}\left\{\int_{{\mathbb{R}}\backslash \{ 0\}}\left[\varphi\left(x+z\right)-\varphi(x)-z\varphi^{'}(x){I}_{\{|z|<-\sigma(x)\}}(z)\right]\nu_{\alpha,-\beta}(dz)\right\}(-\sigma(x))^{\alpha}v(x)dx \\
%&=\int_{{\mathbb{R}}\backslash \{ 0\}}\left\{\int_{-\infty}^{\infty}\left[\varphi\left(x+z\right)-\varphi(x)-z\varphi^{'}(x){I}_{\{|z|<-\sigma(x)\}}(z)\right](-\sigma(x))^{\alpha}v(x)dx \right\}\nu_{\alpha,-\beta}(dz)\\
%&=\int_{{\mathbb{R}}\backslash \{ 0\}}\left\{\int_{-\infty}^{\infty}\varphi\left(x\right)(-\sigma(x-z))^{\alpha}v(x-z)dx-\int_{-\infty}^{\infty}\varphi(x)(-\sigma(x))^{\alpha}v(x)dx
%\right.\\ &\phantom{=\;\;}\left.-\int_{-\infty}^{\infty}z\varphi^{'}(x)(-\sigma(x))^{\alpha}v(x){I}_{\{|z|<-\sigma(x)\}}(z)dx \right\}\nu_{\alpha,-\beta}(dz)\\
&=\int_{{\mathbb{R}}\backslash \{ 0\}}\left\{\int_{-\infty}^{\infty}\varphi\left(x\right)(-\sigma(x-z))^{\alpha}v(x-z)dx-\int_{-\infty}^{\infty}\varphi(x)(-\sigma(x))^{\alpha}v(x)dx \right.\\ &\phantom{=\;\;}\left.+\int_{-\infty}^{\infty}z\varphi(x){I}_{\{|z|<-\sigma(x)\}}(z)\left[(-\sigma(x))^{\alpha}v(x)\right]^{'}dx\right\}\nu_{\alpha,-\beta}(dz)\\
&+\int_{{\mathbb{R}}\backslash \{ 0\}}z|z|^{\alpha}\varphi\left(\sigma^{-1}(-|z|)\right)v\left(\sigma^{-1}(-|z|)\right)\nu_{\alpha,-\beta}(dz) \\
&=\left\langle \varphi(x), \int_{{\mathbb{R}}\backslash \{ 0\}}\left[(-\sigma(x-z))^{\alpha}v(x-z)-(-\sigma(x))^{\alpha}v(x)+\left[(-\sigma(x))^{\alpha}v(x)\right]^{'}z{I}_{\{|z|<-\sigma(x)\}}(z)\right]\nu_{\alpha,-\beta}(dz)\right\rangle \\
&+\beta C_{\alpha}\left\langle \varphi(x), v(x)\sigma^{'}(x) \right\rangle,
\end{aligned}
\end{equation*}
then the adjoint operator of $\widetilde{Q}$ is
\begin{equation}
\label{q2}
\begin{aligned}
&\widetilde{Q}^{*}v(x)=\beta C_{\alpha}v(x)\sigma^{'}(x) \\
&+\int_{{\mathbb{R}}\backslash \{ 0\}}\left[(-\sigma(x-z))^{\alpha}v(x-z)-(-\sigma(x))^{\alpha}v(x)+\left[(-\sigma(x))^{\alpha}v(x)\right]^{'}z{I}_{\{|z|<-\sigma(x)\}}(z)\right]\nu_{\alpha,-\beta}(dz)\\
&=\beta C_{\alpha}v(x)\sigma^{'}(x) \\
&+\int_{{\mathbb{R}}\backslash \{ 0\}}\left[(-\sigma(x+z))^{\alpha}v(x+z)-(-\sigma(x))^{\alpha}v(x)-\left[(-\sigma(x))^{\alpha}v(x)\right]^{'}z{I}_{\{|z|<-\sigma(x)\}}(z)\right]\nu_{\alpha,\beta}(dz).
\end{aligned}
\end{equation}
Combined with \eqref{q1} and \eqref{q2}, we have
\begin{equation*}
\begin{aligned}
\widetilde{Q}^{*}v(x)&=\int_{{\mathbb{R}}\backslash \{ 0\}}\left[|\sigma(x+z)|^{\alpha}v(x+z)-|\sigma(x)|^{\alpha}v(x)-\left[|\sigma(x)|^{\alpha}v(x)\right]^{'}z{I}_{\{|z|<|\sigma(x)|\}}(z)\right]\tilde{\nu}_{\alpha,\beta}(dz)\\
&+\beta C_{\alpha}v(x)\sigma^{'}(x).
\end{aligned}
\end{equation*}
\end{proof}

By Lemma 2, we can derive the FPE in case of multiplicative asymmetric $\alpha$-stable noise.
\begin{theorem}
Under Hypotheses $\bf{H1}$-$\bf{H4}$, the Fokker-Planck equation for the SDE \eqref{ori} in the case that $L$ is an asymmetric one-dimensional $\alpha$-stable process is
\begin{equation}
\label{densityAsy}
p_t=B^{*}p, ~~~~p(x,0)=\delta(x-x_0),
\end{equation}
where
\begin{equation*}
\begin{aligned}
B^{*}p(x,t)&=-\partial_{x}(f(x)p(x,t))+\frac{1}{2}\partial_{xx}(g^{2}(x)p(x,t))+\widetilde{Q}^{*}p(x,t).\\
\end{aligned}
\end{equation*}
\end{theorem}

\subsection{Numerical method}
Here, we present the numerical schemes for the absorbing boundary condition.

Let $u(x,t)=|\s(x)|^{\a}p(x,t)$, then we have
\begin{equation}
\label{sum}
\begin{aligned}
&\ \int_{{\mathbb{R}}\backslash \{ 0\}}\left[|\sigma(x+z)|^{\alpha}p(x+z,t)-|\sigma(x)|^{\alpha}p(x,t)-\partial_{x}\left[|\sigma(x)|^{\alpha}p(x,t)\right]z{I}_{\{|z|<|\sigma(x)|\}}(z)\right]\tilde{\nu}_{\alpha,\beta}(dz)\\
= &\ \int_{{\mathbb{R}}\backslash \{  0\}}\left[u(x+z,t)-u(x,t)-\partial_{x}(u(x,t)) z{I}_{\{|z|<|\sigma(x)|\}}(z)\right]\tilde{\nu}_{\alpha,\beta}(dz)\\
=:&\ I.
\end{aligned}
\end{equation}
In the sequel, we will discuss \eqref{sum} in four cases.

Case $1$: $ \sigma(x)>0$ and $\sigma(x)>1-x$. We have
%\begin{equation*}
%\begin{aligned}
%I=&\ \int_{{\mathbb{R}}\backslash \{  0\}}\left[u(x+z,t)-u(x,t)-\partial_{x}(u(x,t))z{I}_{\{|z|<|\sigma(x)|\}}(z)\right]{\nu}_{\alpha,-\beta}(dz)\\
%=&\ {\rm{{\mathbf{P.V}.}}}~  C_{p}(\beta)\int_{\mathbb{R}\backslash \{  0\}}\frac{u(x+z,t)-u(x,t)}{|z|^{1+\alpha}}dz\\
%&\ -\beta C_{\alpha}\int_{\mathbb{R}^{+}}\left[u(x+z,t)-u(x,t)-\partial_{x}(u(x,t))zI_{|z|<\sigma(x)}(z)\right]\frac{1}{z^{1+\alpha}}dz \\
%=:&\ C_{p}(\beta)N_1-\beta C_{\alpha}N_2.
%\end{aligned}
%\end{equation*}
%where
%\begin{equation*}
%\begin{aligned}
%N_1&=\int_{-1-x}^{1-x}\frac{u(x+z,t)-u(x,t)}{|z|^{1+\alpha}}dz-\frac{u(x,t)}{\alpha}\left[(1-x)^{-\alpha}+(1+x)^{-\alpha}\right],\\
%N_2&= \int_{0}^{1-x}\big[u(x+z,t)-u(x,t)-z\partial_{x}(u(x,t))\big]\frac{1}{z^{1+\alpha}}dz-\frac{u(x,t)}{\alpha}(1-x)^{-\alpha}\\
%&\ -\frac{\partial_{x}(u(x,t))}{1-\alpha}\left[(\sigma(x))^{1-\alpha}-(1-x)^{1-\alpha}\right].
%\end{aligned}
%\end{equation*}
%
%Therefore,
\begin{equation*}
\begin{aligned}
I=&\ C_{p}(\beta)\left\{\int_{-1-x}^{1-x}\frac{u(x+z,t)-u(x,t)}{|z|^{1+\alpha}}dz-\frac{u(x,t)}{\alpha}\left[(1-x)^{-\alpha}+(1+x)^{-\alpha}\right]\right\}\\
&\ -\beta C_{\alpha}\left\{\int_{0}^{1-x}\big[u(x+z,t)-u(x,t)-z\partial_{x}(u(x,t))\big]\frac{1}{z^{1+\alpha}}dz-\frac{u(x,t)}{\alpha}(1-x)^{-\alpha}\right.\\ & \phantom{=\;\;}\left.
-\frac{\partial_{x}(u(x,t))}{1-\alpha}\left[(\sigma(x))^{1-\alpha}-(1-x)^{1-\alpha}\right]\right\}.
\end{aligned}
\end{equation*}

Case $2$: $ \sigma(x)>0$ and $\sigma(x)<1-x$. We have
\begin{equation*}
\begin{aligned}
I=&\ C_{p}(\beta)\left\{\int_{-1-x}^{1-x}\frac{u(x+z,t)-u(x,t)}{|z|^{1+\alpha}}dz-\frac{u(x,t)}{\alpha}\left[(1-x)^{-\alpha}+(1+x)^{-\alpha}\right]\right\}\\
&\ -\beta C_{\alpha}\left\{\int_{0}^{\sigma(x)}\big[u(x+z,t)-u(x,t)-z\partial_{x}(u(x,t))\big]\frac{1}{z^{1+\alpha}}dz
-\frac{u(x,t)}{\alpha}(1-x)^{-\alpha}
\right.\\ &\phantom{=\;\;}\left.+\int_{\sigma(x)}^{1-x}\left[u(x+z,t)-u(x,t)\right]\frac{1}{z^{1+\alpha}}dz\right\}.
\end{aligned}
\end{equation*}

Case $3$: $ \sigma(x)<0$ and $-\sigma(x)>1-x$. We have

\begin{equation*}
\begin{aligned}
I=&\ C_{n}(\beta)\left\{\int_{-1-x}^{1-x}\frac{u(x+z,t)-u(x,t)}{|z|^{1+\alpha}}dz-\frac{u(x,t)}{\alpha}\left[(1-x)^{-\alpha}+(1+x)^{-\alpha}\right]\right\}\\
&\ +\beta C_{\alpha}\left\{\int_{0}^{1-x}\big[u(x+z,t)-u(x,t)-z\partial_{x}(u(x,t))\big]\frac{1}{z^{1+\alpha}}dz-\frac{u(x,t)}{\alpha}(1-x)^{-\alpha}\right.\\ &\phantom{=\;\;}\left.
-\frac{\partial_{x}(u(x,t))}{1-\alpha}\left[(-\sigma(x))^{1-\alpha}-(1-x)^{1-\alpha}\right]\right\}.
\end{aligned}
\end{equation*}

Case $4$: $ \sigma(x)<0$ and $-\sigma(x)<1-x$. We have
\begin{equation*}
\begin{aligned}
I =&\ C_{n}(\beta)\left\{\int_{-1-x}^{1-x}\frac{u(x+z,t)-u(x,t)}{|z|^{1+\alpha}}dz-\frac{u(x,t)}{\alpha}\left[(1-x)^{-\alpha}+(1+x)^{-\alpha}\right]\right\}\\
&\ +\beta C_{\alpha}\left\{\int_{0}^{-\sigma(x)}\big[u(x+z,t)-u(x,t)-z\partial_{x}(u(x,t))\big]\frac{1}{z^{1+\alpha}}dz
-\frac{u(x,t)}{\alpha}(1-x)^{-\alpha}
\right.\\ &\ \phantom{=\;\;}\left.+\int_{-\sigma(x)}^{1-x}\left[u(x+z,t)-u(x,t)\right]\frac{1}{z^{1+\alpha}}dz\right\}.
\end{aligned}
\end{equation*}
Let
\bess
   C_{1,\beta}(x)&=& C_p(\beta) 1_{\{\s(x)>0\}}(x)+C_n(\beta)1_{\{\s(x)<0\}}(x),\\
   C_{2,\beta}(x)&=&\beta C_\a 1_{\{\s(x)<0\}}(x)- \beta C_\a1_{\{\s(x)>0\}}(x),
\eess
then $C_{1,\beta}(x)+C_{2,\beta}(x)= C_{1,-\beta}(x)$.

Therefore, for $|\s(x)|>1-x$, we get
\bess
I&=&C_{1,\beta}(x)\int_{-1-x}^{1-x} \frac{u(x+z,t)-u(x,t)}{|z|^{1+\a}}\dz +C_{2,\beta}(x)\int_0^{1-x} \frac{u(x+z,t)-u(x,t)-z\partial_{x}(u(x,t))}{z^{1+\a}} \dz \\
&& -\frac{C_{1,-\beta}(x)u(x,t)}{\a}(1-x)^{-\a}  -\frac{C_{1,\beta}(x)u(x,t)}{\a}(1+x)^{-\a} \\
&& -\frac{C_{2,\beta}(x)\partial_{x}(u(x,t))}{1-\a} \big[|\s(x)|^{1-\a}-(1-x)^{1-\a}\big],    \\
%&=&C_{1,\beta}(x)\bigg\{\int_{-1-x}^{1-x} \frac{u(x+z,t)-u(x,t)}{|z|^{1+\a}}\dz
%-\frac{u(x,t)}{\a}\left[(1-x)^{-\a}+(1+x)^{-\a}\right]\bigg\} \\
%&&+C_{2,\beta}(x)\bigg\{\int_0^{1-x} \frac{u(x+z,t)-u(x,t)-\partial_{x}(u(x,t))z}{z^{1+\a}} \dz -\frac{u(x,t)}{\a}(1-x)^{-\a} \\
%&&-\frac{\partial_{x}(u(x,t))}{1-\a} \big[|\s(x)|^{1-\a}-(1-x)^{1-\a}\big]   \bigg\},
\eess
and for $|\s(x)|<1-x$, we have
\bess
I&=&C_{1,\beta}(x)\int_{-1-x}^{1-x} \frac{u(x+z,t)-u(x,t)}{|z|^{1+\a}}\dz+C_{2,\beta}(x)\bigg\{\int_{0}^{|\s(x)|} \frac{u(x+z,t)-u(x,t)-\partial_{x}(u(x,t))z}{z^{1+\a}} \dz \\
&&+\int_{|\s(x)|}^{1-x}  \frac{u(x+z,t)-u(x,t)}{z^{1+\a}} \dz  \bigg\}-\frac{C_{1,-\beta}(x)u(x,t)}{\a}(1-x)^{-\a}-\frac{C_{1,\beta}(x)u(x,t)}{\a}\left(1+x\right)^{-\a}. \\
%&=&C_{1,\beta}(x)\bigg\{\int_{-1-x}^{1-x} \frac{u(x+z,t)-u(x,t)}{|z|^{1+\a}}\dz -\frac{u(x,t)}{\a}[(1-x)^{-\a}+(1+x)^{-\a}]\bigg\} \\
%&&+C_{2,\beta}(x)\bigg\{\int_{0}^{|\s(x)|} \frac{u(x+z,t)-u(x,t)-\partial_{x}(u(x,t))z}{z^{1+\a}} \dz - \frac{u(x,t)}{\a(1-x)^\a} \\
%&&+\int_{|\s(x)|}^{1-x}  \frac{u(x+z,t)-u(x,t)}{z^{1+\a}} \dz  \bigg\}.
\eess
Then Eq.~(\ref{densityAsy}) becomes
\begin{equation*}
\begin{aligned}
u_t &=-|\s(x)|^{\a}\partial_{x}\left(\frac{f(x)}{|\s(x)|^\a}u(x,t)\right)+\frac{1}{2}|\s(x)|^{\a}\partial_{xx}\left(\frac{g^{2}(x)}{|\s(x)|^\a}u(x,t)\right) +|\s(x)|^\a I +\beta C_\a \s'(x) u(x,t) \\
&=\frac{1}{2}g^{2}(x)\partial_{xx}(u(x,t))+M(x)\partial_{x}(u(x,t))+\left(N(x)+\beta C_\a \s'(x)\right)u(x,t)+|\s(x)|^\a I.
\end{aligned}
\end{equation*}
By the absorbing boundary condition, for $|\s(x)|>1-x$, we have
\bear
\label{density3}
u_t &=&\frac{1}{2}g^{2}(x)\partial_{xx}(u(x,t))+\hat{M}(x)\partial_{x}(u(x,t))+\hat{N}(x)u(x,t) \nonumber \\
    &&+|\s(x)|^\a C_{1,\beta}(x)\int_{-1-x}^{1-x} \frac{u(x+z,t)-u(x,t)}{|z|^{1+\a}}\dz \\
    &&+|\s(x)|^\a C_{2,\beta}(x)\int_0^{1-x} \frac{u(x+z,t)-u(x,t)-\partial_{x}(u(x,t))z}{z^{1+\a}} \dz, \nonumber
\enar
and for $|\s(x)|<1-x$, we have
\begin{equation}
\label{density5}
\begin{aligned}
u_t =&\ \frac{1}{2}g^{2}(x)\partial_{xx}(u(x,t))+ \hat{M}(x)\partial_{x}(u(x,t))+\hat{N}(x)u(x,t)\\
 &\ +|\s(x)|^\a C_{1,\beta}(x)\int_{-1-x}^{1-x} \frac{u(x+z,t)-u(x,t)}{|z|^{1+\a}}\dz  \\
&\ +|\s(x)|^\a C_{2,\beta}(x)\left\{\int_{0}^{|\s(x)|} \frac{u(x+z,t)-u(x,t)-\partial_{x}(u(x,t))z}{z^{1+\a}} \dz
\right.\\ &\phantom{=\;\;}\left.+\int_{|\s(x)|}^{1-x}  \frac{u(x+z,t)-u(x,t)}{z^{1+\a}} \dz  \right\},
\end{aligned}
\end{equation}
where
\bess
     \hat{M}(x) =
    \begin{cases}
       M(x) - |\s(x)|^\a \frac{C_{2,\beta}(x)}{1-\a} \big[|\s(x)|^{1-\a}-(1-x)^{1-\a}\big],   &\text{ $ |\s(x)|>1-x$,}\\
       M(x),  \;  &\text{ $ |\s(x)|<1-x$,}
    \end{cases}
\label{m1}
\eess
and
\begin{equation*}
   \hat{N}(x)= N(x)+\beta C_\a \s'(x) -|\s(x)|^\a \frac{C_{1,-\beta}(x)}{\a}(1-x)^{-\a} -|\s(x)|^\a \frac{C_{1,\beta}(x) }{\a}(1+x)^{-\a}.
\end{equation*}
Denote $U_j$ as the numerical solution of $u$ for Eq.~(\ref{density3}) and (\ref{density5}) at $(x_j,t)$,
where $x_j=jh$ for $-J< j< J$ and $h = \frac1J$. The unknowns vector by $\mathbf{U}:= (U_{-J+1}, U_{-J+2}, \cdots, U_{J-1})^T$. We approximate the diffusion term by the second-order central differencing and the first-order derivatives by the first-order upwind scheme. Then we can discretize the non-integral terms in the right-hand side (RHS) of Eqs.~\eqref{density3} and \eqref{density5} as
\bess
\left(A\mathbf{U}\right)_j:= C_h\frac{U_{j+1}-2U_j+U_{j-1}}{h^2} + \hat{M}(x_j) \delta_u U_j
     +\hat{N} U_j,\label{eq.B}
\eess
where
\begin{equation*}
   C_h=\frac{g^2(x)}{2}-\frac{ C_\a|\s(x)|^\a}{2} \zeta(\a-1)h^{2-\a}.
\end{equation*}
In fact, the second term in $C_h$ is the leading-order correction term for the trapezoidal rules of the singular integrals in Eqs.~(\ref{density3}) and (\ref{density5}), and $\zeta$ is the Riemann zeta function. The integrals in Eqs.~(\ref{density3}) and (\ref{density5})
are approximated by the trapezoidal rule
\bess
(B\mathbf{U})_j :=
    \begin{cases}
     & \left|\s(x_j)\right|^\a C_{1,\beta}(x_j) h\sum_{k=1-J, k\neq j}^{J-1} \frac{U_k-U_j}{|x_k-x_j|^{1+\a}} \\
    &+ \left|\s(x_j)\right|^\a C_{2,\beta}(x_j) h\sum\!{'}_{k=j+1}^{J} \frac{U_k-U_j-(x_k-x_j)\frac{U_{j}-U_{j-1}}{h}}{|x_k-x_j|^{\a+1}},
     \quad \text{ for $|\s(x_j)|+x_j>1$,}  \\
     &  \left|\s(x_j)\right|^\a C_{1,\beta}(x_j) h\sum_{k=1-J, k\neq j}^{J-1} \frac{U_k-U_j}{|x_k-x_j|^{1+\a}} + |\s(x_j)|^\a C_{2,\beta}(x_j)h\\
    &\left[\sum\!{'}_{k=j+1}^{j+m} \frac{U_k-U_j-(x_k-x_j)
    \frac{U_{j}-U_{j-1}}{h}}{|x_k-s_j|^{\a+1}}+\sum_{k=j+m}^{J} \frac{U_k-U_j}{|x_k-x_j|^{1+\a}}\right],
      ~~\text{for $|\s(x_j)|+x_j<1$,} \\
    \end{cases}
\eess
where $m=\left[\frac{|\s(x_j)|}{h}\right]$ means the integer portion of the value $\frac{|\s(x_j)|}{h}$, the summation symbol $\sum$ means the terms of both end indices are multiplied by $\frac12$,
$\sum\!{'}$ means that only the term of the top index is multiplied by $\frac12$.

Set $\tilde{R}=A+B$, then the semi-discrete scheme for Eqs.~(\ref{density3}) and (\ref{density5}) becomes
\begin{equation}
\frac{{\rm d}U_j}{{\rm d}t} =  (\tilde{R}\mathbf{U})_j,
\label{eq.FPEsd}
\end{equation}
for $-J+1 \leq j \leq J-1$.

In the next, we will show the semi-discrete scheme \eqref{eq.FPEsd} satisfies discrete maximal principle for the absorbing condition. Let
\begin{equation*}
\begin{aligned}
    &I_h = \{j\in \mathbb{Z}:|jh|<1\}, \\
    &I_{h,T}=I_h \times (0,T], \\
&\partial I_{h,T}= \mathbb{Z} \times [0,T] \setminus I_{h,T},
\end{aligned}
\end{equation*}
where $T>0$ is the final time.

\begin{prop}[Maximum principle for the absorbing condition]
  Assume $ \hat{N}(x) \leq 0$, $C_{2,\beta}(x) \geq 0$  for $ x\in (-1,1)$.
  \begin{enumerate}[label=(\roman*)]
   \item\label{prop.1}  If
\begin{equation*}
   \frac{{\rm d}U_j}{{\rm d}t} -  (\tilde{R}\mathbf{U})_j \leq 0, ~~
   \text{  for } (j,t)\in I_{h,T},  \quad
\label{eq.mp1}
\end{equation*}
and
\begin{equation*}
 \quad U_j = 0, ~~\text{ for } |j|\geq J,
\end{equation*}
then we have
\begin{equation*}
	\max_{(j,t)\in \mathbb{Z}\times [0, T]} U_j(t) = \max_{(j,t)\in \partial I_{h,T}} U_j(t).
\end{equation*}
    \item\label{prop.2}

 If
\begin{equation*}
   \frac{{\rm d}U_j}{{\rm d}t} -  (\tilde{R}\mathbf{U})_j \geq 0,
   ~~\text{  for } (j,t) \in I_{h,T},
\label{eq.mp3}
\end{equation*}
and
\begin{equation*}
\quad U_j = 0, ~~~\text{ for } |j|\geq J,
\end{equation*}
then we have
\begin{equation*}
	\min_{(j,t)\in \mathbb{Z}\times [0,T]} U_j(t) = \min_{(j,t)\in \partial I_{h,T}} U_j(t).
\end{equation*}
\end{enumerate}
\end{prop}
\begin{proof}
See Appendix A.3.
\end{proof}

\begin{remark}
The theoretical analysis of the weak and strong maximum principles for nonlocal Waldenfels
operator has been studied \cite{qiao16}. The stability and convergence of the schemes follows from the maximum principle and the
 Lax equivalence theorem due to the linearity of the equations (\ref{density3}) and (\ref{density5}).
\end{remark}

\subsection{Numerical experiment}
Here we present an example to illustrate this numerical scheme for asymmetric FPE.

\begin{example}
Consider the Langevin equation is of the form
\begin{equation}
\label{example2}
dX_t=f(X_t)dt+g(X_t)d B_t+\sigma(X_t)dL^{\alpha,\beta}_{t}, ~~~~X_0=0,
\end{equation}
where the function $f$ is determined by a function $V$, so that $ f=-\partial V(x,t)/\partial x$. Here we consider a motion in the time-independent parabolic potential
\begin{equation}
V(x)=0.1x^2.
\end{equation}
The first noise term  is a Gaussian noise with intensity $g(x)$. The second noise term is an asymmetric $\alpha$-stable noise with intensity $\sigma(x)=\pi+\arctan x$. Perhaps our method may be work for some problems from the book \cite{wt}, but we only show a simple example here.
\end{example}

\befig[h]
\begin{center}
\includegraphics*[width=0.7\linewidth]{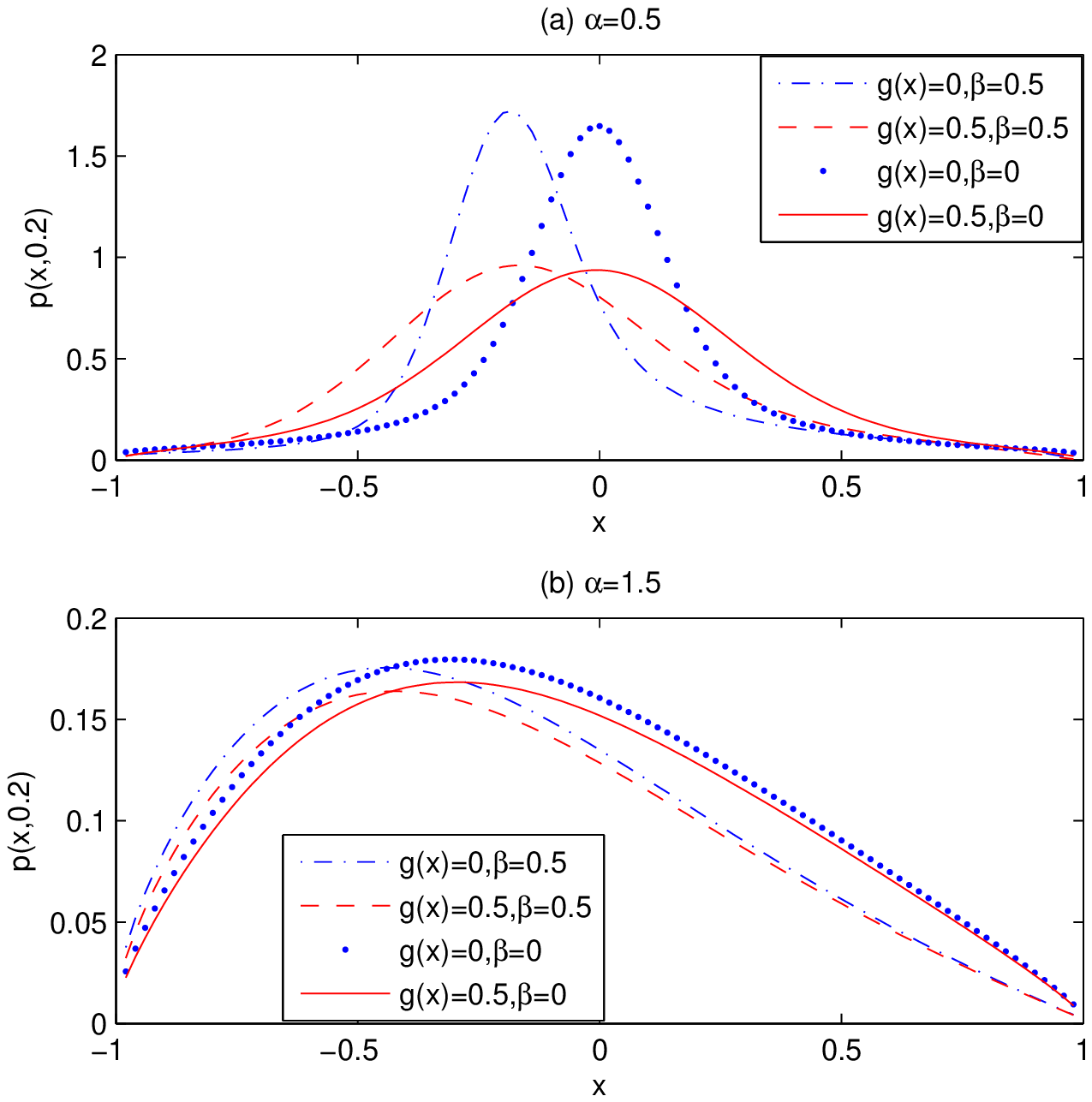}
\end{center}
\caption{The FPE driven by asymmetric multiplicative $\alpha$-stable L\'evy noises with $\sigma(x)=\pi+\arctan x, f(x)=-0.2x$ and $D=(-1,1)$  at time $t=0.2$ for different $\beta=0.5, 0$ and $g(x)=0, 0.5$.
 The initial condition is $p(x,0.01) = \sqrt{\frac{40}{\pi}}e^{-40x^2}$. (a) $\alpha = 0.5$; (b) $\alpha = 1.5$.}
\label{multiFPEAsy1}
\enfig

Here we use the backward Euler scheme for the time integration, which satisfies the maximum principle. We take the initial condition $p(x,0.01) = \sqrt{\frac{40}{\pi}}e^{-40x^2}$ and $D=(-1,1)$ for different skewness parameter $\beta$ and Gaussian noise intensity $g(x)$. Fig.~\ref{multiFPEAsy1} shows the solution of FPE with absorbing condition at time $t=0.2$. Moreover,  we find that the effects of Gaussian noise for the solution of FPE in the case of asymmetric $\alpha$-stable noise are similar to that of symmetric $\alpha$-stable noise. Besides, the skewness parameter $\beta$ has greater effects for the solution as $\a$ is smaller.

\section{Application to nonlinear filtering problem}
\subsection{Analysis of nonlinear filtering problem}
In this subsection, we derive the strong form of Zakai equation under  multiplicative $\alpha$-stable noise. Consider the following signal-observation systems on $\mathbb{R}^2$
\begin{equation}
\label{so}
\left\{
\begin{aligned}
dX_t&=f(X_t)dt+\sigma({X_t})dL_t^{{\alpha}}, \\
dY_t&=f_2(t, X_t)dt+\int_{\mathbb{R}}\gamma(t,X_{t-},z)\widetilde{N}(dt,dz),
\end{aligned}
\right.
\end{equation}
where $f$ and $f_2$ are given deterministic vector fields, and $\sigma$ is called the noise intensity. The one-dimensional L\'evy process $L^{{\alpha}}$ has generating triplet $(0,0,{\nu}_{\alpha})$.  The predictable compensator of the Poisson random measure $N$ is given by $dt\otimes \nu(du)$, where $\nu$ is a L\'evy measure. Moreover, $\widetilde{N}\left((0,t],du\right)=N((0,t],du)-t\nu(du)$.

We need to add two new assumptions  on the drift and  diffusion coefficients.
\par
{\bf Hypothesis  H5. }
The nonlinear term $f_2$ is bounded and Lipschitz, i.e., there exists a positive constant $C_1$ such that for any $t\in[0,T]$ and all $x,y\in\R$,
 \begin{equation*}
 \sup_{x\in \mathbb{R}}|f_2(t,x)|\leq C_1,
 \end{equation*}
 and
 \begin{equation*}
 |f_2(t,x)-f_2(t,y)|\leq C_1|x-y|.
 \end{equation*}
\par
{\bf Hypothesis  H6. }
There exists a positive constant $C_2$ such that for all $x,y\in\R$,
\begin{equation*}
\int_{0}^{T}\int_{\mathbb{R}}|\gamma(t, x, z)|^{2}\nu(dz)dt\leq C_2(1+x^2),
\end{equation*}
and
\begin{equation*}
\int_{0}^{T}\int_{\mathbb{R}}|\gamma(t, x, z)-\gamma(t, y, z)|^2\nu(dz)dt\leq C_2|x-y|^2.
\end{equation*}
\begin{remark}
Under Hypotheses $\bf{H1}$-$\bf{H3}$, $\bf{H5}$ and $\bf{H6}$, the signal-observation system \eqref{so} exists a unique solution.
\end{remark}
Let
\begin{equation*}
\mathcal{Y}_t=\mathcal{R}(Y_s:0\leq s\leq t)\vee \mathcal{N},
\end{equation*}
where $\mathcal{N}$ is the  collection of all $\mathbb{P}$ -negligible sets of $(\Omega,\mathcal{F})$.

The filtering problem aims at determining the conditional distribution of the signal $X_t$ at time $t$, given the information accumulated by observing $Y_t$ in the time interval $[0,t]$. What we are interested in is deriving the strong form of Zakai equation under multiplicative $\alpha$-stable noise. Before that, we present the change of probability measure for It\^o-L\'evy process, which comes from \cite[Theorem 1.31]{os}.
\begin{theorem}
Let $M$ be an one-dimensional It\^o-L\'evy processes of the form
\begin{equation*}
dM(t)=\alpha(t,\omega)dt+\sigma(t,\omega)dB(t)+\int_{\mathbb{R}}\gamma(t,z,\omega)\widetilde{N}(dt,dz), ~~0\leq t\leq T.
\end{equation*}
Assume there exist a predictable $\mathbb{R}$-valued process $u=\{u(t),t\in[0,T]\}$ and a family of predictable $\mathbb{R}$-valued processes $\theta(\cdot,z)=\{\theta(t,z),t\in[0,T]\}, z\in \mathbb{R}\backslash \{ 0\}$, such that
\begin{equation}
\label{op}
\sigma(t)u(t)+\int_{\mathbb{R}} \gamma(t,z)\theta(t,z)\nu(dz)=\alpha(t), ~for~ almost ~all ~(t, \omega) \in [0,T]\times \Omega,
\end{equation}
and the process
\begin{equation*}
\begin{aligned}
Z_{1}(t)&:=\exp\left\{-\int^{t}_{0}u(s)dB(s)-\frac{1}{2}\int^{t}_{0}u^{2}(s)ds+\int^{t}_{0}\int_{\mathbb{R}}\ln(1-\theta(s,z))\widetilde{N}(ds,dz)
\right.\\ &\phantom{=\;\;}\left.
+\int^{t}_{0}\int_{\mathbb{R}}\left[\ln(1-\theta(s,z))+\theta(s,z)\right]\nu(dz)ds\right\}, ~~0\leq t \leq T
\end{aligned}
\end{equation*}
is well defined and satisfies
\begin{equation*}
E[Z_1(T)]=1.
\end{equation*}
Define the probability measure $Q$ on $\mathcal{F}_{T}$ by $dQ=Z_{1}(T)dP$, then $M$ is a local martingale with respect to $Q$.
\end{theorem}
\begin{remark}
In Theorem 3, we require that for any $z\in \mathbb{R}\backslash \{ 0\}$, $\theta(\cdot,z)$ is a predictable process such that $\theta(t,z)<1$,
\begin{equation*}
\int_{0}^T \int_{\mathbb{R}} \left\{|\ln\left(1-\theta(t,z)\right)|^2+\theta^2(t,z)\right\}\nu(dz)dt<\infty,
\end{equation*}
and $u$ is a predictable process such that
\begin{equation*}
\int_{0}^Tu^{2}(t)<\infty.
\end{equation*}

\end{remark}
\begin{corollary}
Define the process $B_{Q}(t)$ and the random measure $\widetilde{N}_{Q}$ by
\begin{equation*}
dB_{Q}(t)=u(t)dt +dB(t),
\end{equation*}
and
\begin{equation*}
\widetilde{N}_{Q}(dt,dz)=\frac{\theta(t,z)}{1-\theta(t,z)}\nu(dz)dt+\frac{1}{1-\theta(t,z)}\widetilde{N}(dt,dz),
\end{equation*}
then under the new probability measure $Q$, $B_{Q}$ is a Brownian motion and $\widetilde{N}_{Q}$ is a martingale-valued measure.
\end{corollary}
\begin{proof}
The result comes directly from \cite[Theorem 1.35 ]{os}.
\end{proof}
\begin{remark}
 In fact, $\widetilde{N}_{Q}(dt,dz)$ is not the $Q$-compensated random measure associated to $N(dt,dz)$.  The compensated random measure $\widetilde{N}^{c}_{Q}(dt,dz)$ associated to $N(dt,dz)$  under the new probability measure $Q$ satisfies the following equation, i.e.,
 \begin{equation*}
 \widetilde{N}^{c}_{Q}(dt,dz)=\theta(t,z)\nu(dz)dt+\widetilde{N}(dt,dz).
 \end{equation*}

\end{remark}
In the following, we will present the property of Girsanov Theorem for It\^o-L\'evy processes.
\begin{lemma}
 Set
\begin{equation*}
\mathbb{H}^{-1}_t=\exp\left\{\int_{0}^{t}\int_{\mathbb{R}}\ln \left(1- \theta(s,z)\right)\widetilde{N}(ds,dz)+\int_{0}^{t}\int_{\mathbb{R}}\left[\ln\left(1-\theta(s,z)\right)+\theta(s,z)\right]\nu(dz)ds\right\},
\end{equation*}
then $\mathbb{H}_t$ satisfies the following  equation
\begin{equation*}
\mathbb{H}_t=1+\int_{0}^{t}\int_{\mathbb{R}}\mathbb{H}_{s-}\theta(s,z)\widetilde{N}_{Q}(ds,dz).
\end{equation*}
\begin{proof}
Set
\begin{equation*}
U(t)=-\int_{0}^{t}\int_{\mathbb{R}}\ln \left(1- \theta(s,z)\right)\widetilde{N}(ds,dz)-\int_{0}^{t}\int_{\mathbb{R}}\left[\ln\left(1-\theta(s,z)\right)+\theta(s,z)\right]\nu(dz)ds.
\end{equation*}
Applying It\^o formula to $e^{U(t)}$, we have
\begin{equation*}
\begin{aligned}
d\mathbb{H}_t:=de^{U(t)}&=\mathbb{H}_t\int_{\mathbb{R}}\left[\frac{1}{1-\theta(t,z)}-1\right]N(dt,dz)
-\mathbb{H}_t\int_{\mathbb{R}}\theta(t,z)\nu(dz)dt\\
&=\int_{\mathbb{R}}\mathbb{H}_{t-}\theta(t,z)\widetilde{N}_{Q}(dt,dz).
\end{aligned}
\end{equation*}
\end{proof}
\end{lemma}
We introduce a new measure $Q$ via $dQ/{d\mathbb{P}}={1}/{\mathbb{H}_T}$ and define
\begin{equation*}
{{P}_t}(\varphi):=\mathbb{E}^{Q}\left[\varphi(X_t)\mathbb{H}_t|\mathcal{Y}_t\right].
\end{equation*}
Now, we are ready to obtain the Zakai equation for ${{P}_t}(\varphi)$.
\begin{theorem}
For $\varphi \in \mathcal{D}({A})$, ${P}_t(\varphi)$ satisfies the following stochastic evolution equation, i.e.,
\begin{equation}
d{P}_t(\varphi)={P}_t({A}\varphi)dt+\int_{\mathbb{R}}
{P}_{t-}\left(\varphi\theta(t,z)\right)\widetilde{N}_{Q}(dt,dz),
\end{equation}
where ${A}$ is the infinitesimal generator  of $X_t$, i.e.,
\begin{equation}
({A}\varphi)(x)=f(x)\varphi^{'}(x)+\int_{{\mathbb{R}}\backslash \{ 0\}}\left[\varphi\left(x+\sigma(x)y \right)-\varphi(x)-\sigma(x)y\varphi^{'}(x){I}_{\{|y|<1\}}(y)\right]\nu^{\alpha}(dy).
\end{equation}
\begin{proof}
Step 1: Applying It\^o formula to $X_t$ and using the mutual independence of $L^{\alpha}_{t}$ and $\widetilde{N}_{Q}(dt,du)$, we have
\begin{equation}
\label{xr}
d\left[\varphi(X_t)\mathbb{H}_t\right]=\varphi(X_{t-})d\mathbb{H}_t+\mathbb{H}_{t-}d\varphi(X_{t}).
\end{equation}
Step 2: Taking the conditional expectation on both sides of \eqref{xr}, we obtain
\begin{equation*}
\begin{aligned}
\mathbb{E}^{Q}\left[\varphi(X_t)\mathbb{H}_t|\mathcal{Y}_t\right]=&\ \mathbb{E}^{Q}\left[\varphi(X_0)|\mathcal{Y}_t\right]
+\mathbb{E}^{Q}\left[\int^{t}_{0}\varphi(X_s)d\mathbb{H}_s|\mathcal{Y}_t\right]
+\mathbb{E}^{Q}\left[\int^{t}_{0}\mathbb{H}_s d\varphi(X_s)|\mathcal{Y}_t\right]\\
=&\ {{P}_0}(\varphi)+\int^{t}_{0}\int_{\mathbb{R}}\mathbb{E}^{Q}\left[\mathbb{H}_s\varphi(X_s)\theta(s,z)|\mathcal{Y}_s\right]\widetilde{N}_{Q}(ds,dz)\\
&\ +\int^{t}_{0}\mathbb{E}^{Q}\left[\mathbb{H}_s{A}\varphi(X_s)|\mathcal{Y}_s\right]ds \\
=&\ {{P}_0}(\varphi)+\int^{t}_{0}\int_{\mathbb{R}}P_{s-}(\varphi\theta(s,z))\widetilde{N}_{Q}(ds,dz)+\int^{t}_{0}{P}_s({A}\varphi)ds.
\end{aligned}
\end{equation*}
Therefore,
\begin{equation}
\label{mar}
d{P}_t(\varphi)={P}_t({A}\varphi)dt+\int_{\mathbb{R}}
{P}_{t-}\left(\varphi\theta(t,z)\right)\widetilde{N}_{Q}(dt,dz).
\end{equation}
\end{proof}
\end{theorem}

In the following, we aim at deriving the strong form of Zakai equation under  multiplicative $\alpha$-stable noise.  By Lemma 1, we know the adjoint operator ${A}$ is
\begin{equation*}
({{A}}^{*}\varphi)(x)=-\left(f(x)\varphi(x)\right)^{'}+\int_{{\mathbb{R}}\backslash \{ 0\}}\left[|\sigma(x+z)|^{\alpha}\varphi(x+z)-|\sigma(x)|^{\alpha}\varphi(x)\right]\nu^{\alpha}(dz).
\end{equation*}
Heuristically, if the unconditional distribution of the signal ${P}_t(\varphi)$ has a density $p(t,x)$ with respect to Lebesgue measure for all $t>0$, i.e.,
${P}_t(\varphi)=\int_{\mathbb{R}}\varphi(x)p(t,x)dx$, then we have the following result. Moreover, the unnormalized density $p(t,x)$ satisfies the following stochastic partial differential equation.
\begin{theorem}
Under Hypotheses $\bf{H1}$-$\bf{H3}$, $\bf{H5}$ and $\bf{H6}$, $p(t,x)$ satisfies the following stochastic partial differential equation, i.e.,
\begin{equation}
\label{dens1}
dp(t,x)={{A}}^{*}p(t,x)dt+\int_{\mathbb{R}}p(t, x)\theta(t,z)\widetilde{N}_{Q}(dt,dz),
\end{equation}
where
\begin{equation}
{{A}}^{*}p(t,x)=-\left(f(x)p(t,x)\right)^{'}+\int_{{\mathbb{R}}\backslash \{ 0\}}\left[|\sigma(x+z)|^{\alpha}p(t,x+z)-|\sigma(x)|^{\alpha}p(t,x)\right]\nu^{\alpha}(dz).
\end{equation}
\begin{proof}
By the definition of ${P}_t(\varphi)$ and \eqref{mar}, we have
\begin{equation*}
\begin{aligned}
{P}_t(\varphi)=&\ \int_{\mathbb{R}}\varphi(x)p(t,x)dx\\
=&\ \int_{\mathbb{R}}\varphi(x)p_0(x)dx+\int_{\mathbb{R}}\left\langle {A}\varphi(x),\int_{0}^{t}p(s,x)ds\right \rangle dx \\
&\ +\int_{\mathbb{R}}\varphi(x)\left[\int_{0}^{t}\int_{\mathbb{R}}p(s,x)\theta(s,z)\widetilde{N}_{Q}(ds,dz)\right]dx.
\end{aligned}
\end{equation*}
Thus, we get
\begin{equation*}
dp(t,x)={A}^{*}p(t,x)dt+\int_{\mathbb{R}}p(t, x)\theta(t,z)\widetilde{N}_{Q}(dt,dz),
\end{equation*}
where $A^{*}$ is the adjoint of the operator $A$.
\end{proof}
\end{theorem}

\subsection{Approximation analysis for nonlinear filtering problem}
 The purpose here is to study one of the Trotter-like product formulas.  Firstly, we develop an infinite partition $0 =t_0<t_1<t_2<\cdots < t_{n}=T$, to be denoted by $\kappa$,  with time increments $\delta_{i}=t_{i}-t_{i-1}$.
\begin{definition}
An admissible sampling procedure relative to the partition $\kappa$ is a family $\{\mathcal{\bar{Y}}^{t_{i+1}}_{t_i}, i\geq 0\}$ of $\sigma$-algebras which satisfy, for all $i\geq 0$
\begin{enumerate}
\item[(i)]$\mathcal{\bar{Y}}^{t_{i+1}}_{t_i}$ is generated by a finite number of random variables;
\item[(ii)]$\mathcal{\bar{Y}}^{t_{i+1}}_{t_i}\subset \mathcal{{Y}}^{t_{i+1}}_{t_i}$.
\end{enumerate}
In addition, we introduce the following notations by
\begin{equation*}
\mathcal{\bar{Y}}^{t_{n}}_{t_m}=:\bigvee^{n-1}_{i=m}\mathcal{\bar{Y}}^{t_{i+1}}_{t_i}, ~~~\mathcal{\bar{Y}}^{t_m}=\mathcal{\bar{Y}}^{{t_m}}_{0}.
\end{equation*}
\end{definition}
By the Kallianpur-Striebel formula, we have
\begin{equation}
\mathbb{E}\big(\phi(X_{t_{i}})|\mathcal{\bar{Y}}^{t_i}\big)
=\frac{\mathbb{E}^{Q}\left[\varphi(X_{t_{i}})\mathbb{H}_{t_{i}}|\mathcal{\bar{Y}}^{t_i}\right]}{\mathbb{E}^{Q}\left[\mathbb{H}_{t_{i}}|\mathcal{\bar{Y}}^{t_i}\right]}.
\end{equation}
The following proposition is from \cite[Proposition 3.2]{fra}.
\begin{prop}
Set
\begin{equation}
\begin{aligned}
M^{t_{i+1}}_{t_{i}}&=:\mathbb{E}^{Q}\left(\mathbb{H}^{t_{i+1}}_{t_{i}}|\mathcal{F}_{t_{i+1}}\bigvee \mathcal{\bar{Y}}^{t_{i+1}}_{t_{i}}\right),\\
N_{i+1}\varphi&=:\mathbb{E}^{Q}\left(\varphi(X_{t_{i+1}})\mathbb{H}^{t_{i+1}}_{t_{i}}|\mathcal{F}_{t_{i}}\bigvee \mathcal{\bar{Y}}^{t_{i+1}}_{t_{i}}\right), \\
(p_{i},\varphi)&=:\mathbb{E}^{Q}\left[\varphi(X_{t_{i}})\mathbb{H}_{t_{i}}|\mathcal{\bar{Y}}^{t_i}\right],
\end{aligned}
\end{equation}
then we have
\begin{equation}
(p_{i+1},\varphi)=\mathbb{E}^{Q}\left[(N_{i+1}\varphi)\mathbb{H}_{t_{i}}|\mathcal{\bar{Y}}^{t_{i+1}}\right].
\end{equation}
\end{prop}
Now we assume $X_{t}=X_{t_{i}}$ for $t_{i}\leq t < t_{i+1}$, then we have
\begin{equation}
\begin{aligned}
M^{t_{i+1}}_{t_{i}}&=\exp\left\{-\int_{\mathbb{R}}\ln (1-\theta(t_i,z))N(\delta_i, dz)-\delta_{i}\int_{\mathbb{R}} \theta(t_i,z)\nu(dz)\right\}\\
&:=\chi_{i}(X_{t_{i}}).
\end{aligned}
\end{equation}

Therefore
\begin{equation}
N_{i+1}\varphi=\chi_{i}(X_{t_{i}})\mathbb{E}^{Q}\left(\varphi(X_{t_{i+1}})|\mathcal{F}_{t_{i}}\right).
\end{equation}
Assume the signal $\{X_t, t\geq 0\}$ is a diffusion process associated with the semi-group $\mathbf{P}_t$, then we have
\begin{equation}
\label{ni2}
N_{i+1}\varphi=\chi_{i}(X_{t_{i}})[{\mathbf{P}_{\delta_{i}}}\varphi](X_{t_{i}}),
\end{equation}
and
\begin{equation}
\begin{aligned}
(p_{i+1},\varphi)&=\mathbb{E}^{Q}\left(\chi_{i}(X_{t_{i}})[{\mathbf{P}_{\delta_{i}}}\varphi](X_{t_{i}})\mathbb{H}_{t_{i}}|\mathcal{\bar{Y}}^{t_{i+1}}\right)\\
&=\left(p_{i}, \chi_{i}[{\mathbf{P}_{\delta_{i}}}\varphi]\right).
\end{aligned}
\end{equation}
Therefore we have
\begin{equation}
p_{i+1}={P^{*}_{\delta}}[\chi_{i}p_{i}].
\end{equation}
Using an implicit Euler scheme, we get the approximation scheme of \eqref{dens1}, i.e.,
\begin{equation}
\label{dsi}
(I-\delta_{i}{A}^{*})p_{i+1}=\chi_{i}p_{i}.
\end{equation}
\begin{remark}
The convergence analysis of this approximation is similar to \cite[Theorem 5.2]{fra}.
\end{remark}
\subsection{An example }
 In this example, we consider the following gradient dynamical system on $\mathbb{R}$ as the signal system,  which represents the potential energy landscape for diffusion of molecules.
\begin{equation}
\label{example12}
\begin{aligned}
dX_t&=f(X_t)dt+\big(2+\sin({X_{t-}})\big)dL_t^{{\alpha}},
\end{aligned}
\end{equation}
where the function $f$ is determined by a energy potential, denoted by $V$ with $V(x,t)=-\frac{1}{2}x^2+\frac{1}{4}x^4$, so that $ f=-\partial V(x,t)/\partial x$. Here we assume the gradient dynamical system is driven by multiplicative  periodically modulated noise, which is widely used to explain the periodic recurrence of the earth's ice age on Earth and describe the evolution of physical phenomena in ideal fluctuating environments \cite{yu}.

In order to track this molecule, we use tunnel electron microscope to receive observation of the molecule \cite{nm}. Here we show a simple example. In this experiment, we assume that the observation system is given by
\begin{equation}
\label{example13}
dY_t=\frac{\cos(X_t)}{2\sqrt{2}}dt+\int_{\mathbb{R}}\cos (X_{t-})e^{-\frac{z^2}{2}} \widetilde{N}(dt,dz),
\end{equation}
where $\widetilde{N}(dt,dz)={N}(dt,dz)-\nu(dz)dt$ with $\nu(dz)=\frac{1}{\sqrt{2 \pi}}e^{-\frac{z^2}{2}}dz$.

Using Theorem 5, the strong form of Zakai equation for the signal-observation system \eqref{example12}-\eqref{example13} is
\begin{equation}
\label{exam3}
dp(t,x)={{A}}^{*}p(t,x)dt+\int_{\mathbb{R}}p(t, x)\theta(t,z)\widetilde{N}_{Q}(dt,dz),
\end{equation}
where
\begin{equation*}
\begin{aligned}
({{A}}^{*}\varphi)(x)=&\ -\left((x-x^3)\varphi(x)\right)^{'}\\
&\ +\int_{{\mathbb{R}}\backslash \{ 0\}}\left[(2+\sin(x+z))^{\alpha}\varphi(x+z)-(2+\sin(x))^{\alpha}\varphi(x)\right]\nu^{\alpha}(dz),
\end{aligned}
\end{equation*}
and
\begin{align*}
\theta(t,z)&=\frac{1}{2}, \\
\widetilde{N}_{Q}(dt,dz)&=2\widetilde{N}(dt,dz)+\nu(dz)dt.
\end{align*}
In Figure 5, we simulate the signal -observation processes with $X_0=0, Y_0=0, \alpha=0.5, \Delta t=5\times 10^{-4}, T=5$. Use an implicit Euler scheme, the solution of Zakai equaiton for Eq.~\eqref{exam3} with $\alpha=0.5, p(x,0) = \sqrt{\frac{40}{\pi}}e^{-40x^2} $ is illustrated in the Figure 6.
\befig[h]
\begin{center}
\includegraphics*[width=0.7\linewidth]{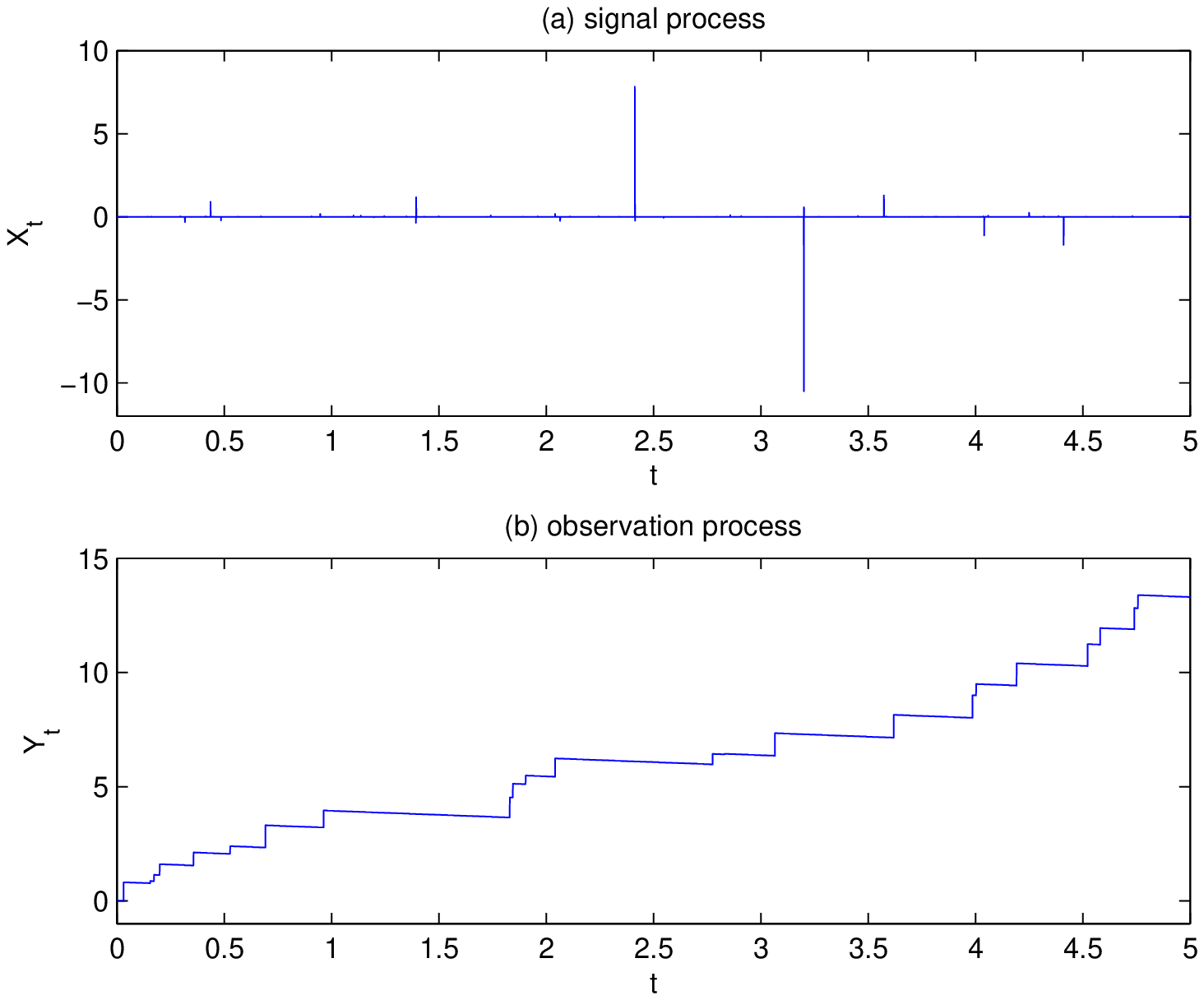}
\end{center}
\caption{The signal-observation systems for \eqref{example12} with $X_0=0, Y_0=0, \alpha=0.5, \Delta t=5\times 10^{-4}, T=5$. }
\label{Filteringalpha05}
\enfig

\befig[h]
\begin{center}
\includegraphics*[width=0.7\linewidth]{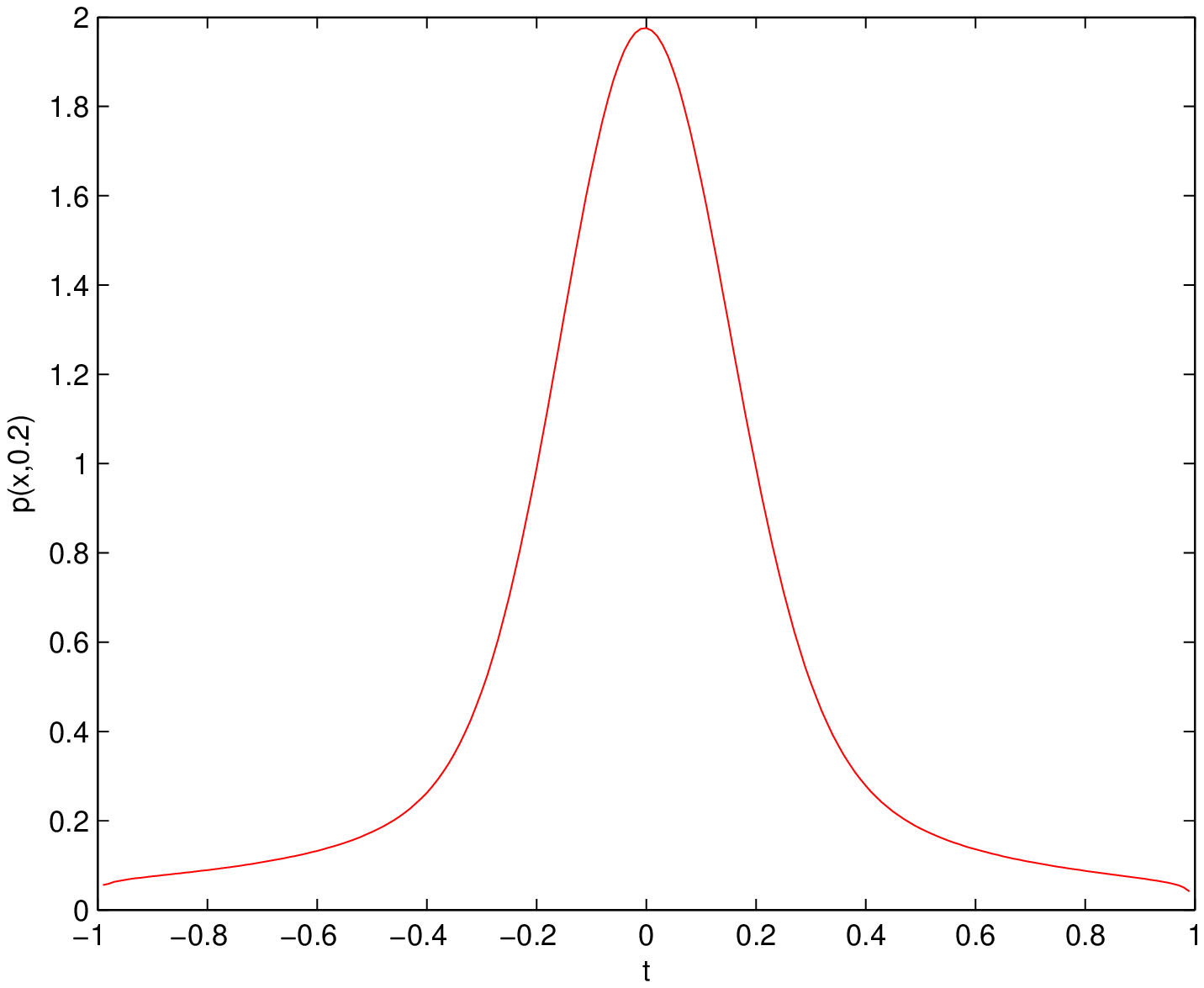}
\end{center}
\caption{The solution of Zakai equaiton for Eq.~\ref{exam3} with $\alpha=0.5, p(x,0) = \sqrt{\frac{40}{\pi}}e^{-40x^2}$. }
\label{filterfpe}
\enfig

\section{Concluding remarks}
\par
In this paper, we have derived the FPEs for stochastic systems with multiplicative $\alpha$-stable noises by the method of adjoint operators, and developed a new numerical scheme with stability and convergence analysis for one-dimensional case. We have also illustrated the scheme with some numerical experiments. Moreover, we have extended these results to the asymmetric case and then have successfully applied to the nonlinear filtering problem. Our results are expected to assist simulation study of, for example, the hydrodynamical systems.

In this paper, it is required that the noise intensity $\sigma$ is strict monotone for the asymmetric case.  The restriction can be relaxed .

Furthermore, we may extend the results to high-dimensional case under certain conditions. This could potentially find applications
in most probable trajectories of anomalous diffusion processes. It is also possible to discuss the connection between the forward and backward Kolmogorov equations. These topics are being studied and will be reported in future works.

\medskip
\textbf{Acknowledgements}.  We would like to thank Xu Sun (Huazhong University of Sciences and Technology, China) and Feng Bao (Florida State University, USA) for helpful discussions. The research of Y. Zhang was supported by the NSFC grant 11901202.  The research
of X. Wang was supported by the NSFC grant 11901159 and the cultivation project of first class subject of Henan University (No. 2019YLZDJL08). The research of J. Duan was supported  by the NSF-DMS no. 1620449 and NSFC grant. 11531006 and
 11771449. The research
of T. Li was supported by the NSFC grant 11801143.

\section{Appendix A. Further Proofs}
A.1. $\mathbf{Proof ~~of ~~Proposition  ~~3.1}$. Suppose $U_j^n$ is the numerical solution for $u$ at $(x_j, t_n)$. For $0<U_j^n \leq M$, where $M$ is the maximum value of the initial probability density. As $0< \alpha< 2$, we have $\zeta(\a-1)\leq 0$. Applying the explicit Euler for \eqref{U}, and taking $c(x)=c(1,\a)|\s(x)|^\a\geq0 $, we get
\begin{equation}
\label{uj}
\begin{aligned}
 U_j^{n+1} =&\ U_j^n-\delt  c(x_j) \zeta(\a-1)h^{-\a} (U_{j+1}^n-2U_j^n+U_{j-1}^n) \\
            &\ - \frac{c(x_j)\delt}{\a}  \bigg[\frac{1}{(1+x_j)^\a} +\frac{1}{(1-x_j)^\a} \bigg]U_j^n
            + c(x_j)\delt h \sum\limits_{k=-J-j, k\neq 0}^{J-j} \frac{U_{j+k}^n-U_j^n}{|x_k|^{1+\a}} \\
   =& \ \Bigg\{1+2\delt c(x_j)\zeta(\a-1) h^{-\a}- \frac{c(x_j)\delt}{\a}  \bigg[\frac{1}{(1+x_j)^\a} +\frac{1}{(1-x_j)^\a}\bigg] \\
            &\ - c(x_j)\delt h \sum\limits_{k=-J-j, k\neq 0}^{J-j} \frac{1}{|x_k|^{1+\a}} \Bigg\} U_j^n \\
   &\ -c(x_j)\delt\zeta(\a-1)h^{-\a} (U_{j-1}^n+U_{j+1}^{n})
            +  c(x_j)\delt h \sum\limits_{k=-J-j, k\neq 0}^{J-j}\frac{ U_{j+k}^n}{|x_k|^{1+\a}}.
\end{aligned}
\end{equation}
Denote
\begin{equation*}
\begin{aligned}
   \hat{ L}:&=-2\delt c(x_j)\zeta(\a-1) h^{-\a}+\frac{c(x_j)\delt}{\a}  \bigg[\frac{1}{(1+x_j)^\a} +\frac{1}{(1-x_j)^\a}\bigg]
            + c(x_j)\delt h \sum\limits_{k=-J-j, k\neq 0}^{J-j}\!\frac{1}{|x_k|^{1+\a}},
    \end{aligned}
\end{equation*}
then by the inequality (\ref{MPcondition0}), for $x_j=jh$,  we have
\begin{equation}
\begin{aligned}
\label{L}
   \hat{ L }&\leq  c(x_j) \delt \left[\int_{-\infty}^{-1-x_j} \frac{\dy}{|y|^{1+\a} } +  \int_{1-x_j}^{\infty} \frac{\dy}{|y|^{1+\a} }
    +\frac{2h}{h^{1+\a}}+\int_{\left(-1-x_j+\frac{h}{2},1-x_j-\frac{h}{2}\right)\backslash (-h,h) } \frac{\dy}{|y|^{1+\a} } \right]\\
    &-2\delt c(x_j)\zeta(\a-1) h^{-\a} \\
    &\leq c(1,\a)|\s(x_j)|^\a \delt \left[\frac{2}{h^\a}+2\int_h^\infty \frac{\dy}{|y|^{1+\a} }\right]-2\delt c(x_j)\zeta(\a-1) h^{-\a}\\
    &\leq \frac{2c(1,\a)|\s(x_j)|^\a \delt }{h^\a}\left(1+\frac{1}{\a}\right)-2\delt c(x_j)\zeta(\a-1) h^{-\a}\\
    &\leq \frac{2c(1, \a)\widetilde{M}^{\a}\delt}{h^{\a}}\left[1+\frac{1}{\a}-\zeta(\a-1)\right].\\
    &\leq 1.
\end{aligned}
\end{equation}
We can rewritten \eqref{uj} as
\begin{equation}
\begin{aligned}
 U_j^{n+1} =&\left(1-\hat{L}\right) U_j^n - c(x_j)\delt\zeta(\a-1)h^{-\a} (U_{j-1}^n+U_{j+1}^{n})
            + c(x_j)\delt h \sum\limits_{k=-J-j, k\neq 0}^{J-j} \frac{U_{j+k}^n}{|x_k|^{1+\a}}\\
             & \leq \left[1-\hat{L}-2c(x_j)\delt \zeta(\a-1) h^{-\a}+c(x_j)\delt h \sum\limits_{k=-J-j, k\neq 0}^{J-j}\!\frac{1}{|x_k|^{1+\a}}  \right]M.
\end{aligned}
\end{equation}

Due to (\ref{L}) and the coefficient $1-\hat{L}$ of $U^{n}_j$ is nonnegative, we see that the coefficients of $U$ are all nonnegative, then we have
\begin{equation*}
\begin{aligned}
U_j^{n+1} \leq & \Bigg\{1- \frac{c(x_j)\delt}{\a}  \bigg[\frac{1}{(1+x_j)^\a} +\frac{1}{(1-x_j)^\a}\bigg] \Bigg\}M \\
  \leq& M.
\end{aligned}
\end{equation*}

A.2. $\mathbf{Proof ~~of ~~Proposition  ~~3.3}$. Set $e^{n}_{j}=U^{n}_{j}-u(x_j, t_n)$, where $u(x_j,t_n)$ is the analytic solution at the point $(x_j,t_n)$, then we have
\begin{equation}
\begin{aligned}
e^{n+1}_{j}-e^{n}_{j}&=&-c(1,\alpha)\zeta(\alpha-1)|\sigma(x_j)|^{\alpha}h^{2-\alpha}\Delta t\frac{e^{n}_{j-1}-2e^{n}_{j}+e^{n}_{j+1}}{h^2} \\
 &&+\Delta tc(1,\alpha)h|\sigma(x_j)|^{\alpha}\sum^{J-j}_{k=-J-j, k\neq 0}\frac{e_{j+k}^n-e_j^n}{|x_k|^{1+\alpha}}
-\Delta t T^{n}_{j},
\end{aligned}
\end{equation}
where
\iffalse
\begin{equation}
\begin{aligned}
&E^{n}_{j}=c(1,\alpha)\zeta(\alpha-1)|\sigma(x_j)|^{\alpha}h^{2-\alpha}\frac{u(x_{j-1},t_n)-2u(x_{j},t_n)+u(x_{j+1},t_n)}{h^2}\\
&-c(1,\alpha)h|\sigma(x_j)|^{\alpha}\sum^{J-j}_{k=-J-j, k\neq 0}\frac{u(x_{j+k},t_n)-u(x_j,t_n)}{|x_k|^{1+\alpha}}+\frac{u(x_j, t_{n+1})-u(x_j, t_{n})}{\Delta t}\\
&\approx c(1,\alpha)\zeta(\alpha-1)|\sigma(x_j)|^{\alpha}h^{2-\alpha}\left[u_{xx}(x_j,t_n)+\frac{1}{12}u_{xxxx}(x_j,t_n)h^2\right]\\
&-c(1,\alpha)h|\sigma(x_j)|^{\alpha}\sum^{J-j}_{k=-J-j, k\neq 0}\frac{u(x_{j+k},t_n)-u(x_j,t_n)}{|x_k|^{1+\alpha}}\\
&+u_{t}(x_j, t_n)+\frac{1}{2}u_{tt}(x_j, t_n)\Delta t.
\end{aligned}
\end{equation}
\fi

\begin{equation}
\begin{aligned}
&T^{n}_{j} = \frac{1}{2}u_{tt}(x_j, t_n)\Delta t+\frac{1}{12}c(1,\alpha)\zeta(\alpha-1)|\sigma(x_j)|^{\alpha}h^{4-\alpha}u_{xxxx}(x_j,t_n)\\
&-\frac{1}{6}c(1,\alpha)\zeta(\alpha-3)|\sigma(x_j)|^{\alpha}h^{4-\alpha}u_{xxxx}(x_j,t_n)
-\frac{B_2}{2}c(1,\alpha)|\sigma(x_j)|^{\alpha}\frac{\partial}{\partial y}\left(\frac{u(x_j+y, t_n)-u(x_j, t_n)}{|y|^{1+\alpha}}\right)\Big|^{y=L-x_j}_{y=-L-x_j}\\
&+c(1,\alpha)|\sigma(x_j)|^{\alpha}\int_{\{-\infty,-L-x_j\}\bigcup \{L-x_j, \infty\}}\frac{u(x_j+y,t_n)-u(x_j,t_n)}{|y|^{1+\alpha}}dy+\cdots,
\end{aligned}
\end{equation}
where $ \zeta(\tau)$ is the Riemann zeta function initially defined for $\mathbb{R}e \tau > 1$ by $\zeta(\tau)=\sum_{k=1}^{\infty}k^{-\tau}$, $B_{2}$ are the Bernoulli numbers \cite{as}.

Obviously, we have
\begin{equation}
|T^{n}_j|\leq O(\Delta t)+ O(h^2)+O(L^{-\alpha}):=\bar{T}.
\end{equation}
Therefore, the truncation error is uniformly bounded.

By the condition \eqref{MPcondition}, we have
\begin{equation}
\max |e^{n+1}_j|\leq \max |e^{n}_j|+\Delta  t \bar{T}\leq n\Delta  t \bar{T},
\end{equation}
Thus, the convergence is proved.

A.3. $\mathbf{Proof ~~of ~~Proposition  ~~4.1}$.

Note that $C_p, C_n\geq 0$, then $C_{1,\beta} \geq 0$.
Assume $(j^*,t^*) \in I_{h,T}$ is the maximum point, then we have
$
\left(A\mathbf{U}\right)_{j^*}:= C_h\frac{U_{j^*+1}-2U_{j^*}+U_{j^*-1}}{h^2} +
\hat{M}(x_{j^*}) \delta_u U_{j^*}+\hat{N} (U_{j^*}-U_J) \leq 0
$
as $U_J=0$.
Due to $C_{1,\beta}, C_{2,\beta} \geq 0$, we have  $(B\mathbf{U})_{j^*} \leq 0$.
By the similar technique from the fifth step to prove Proposition 2  in the reference \cite{Xiao2019}, we get the required result.

\end{document}